\documentclass{amsart}
\usepackage{amsmath}
\usepackage{amssymb}
\usepackage{amsfonts}

\newtheorem{theorem}{Theorem}
\theoremstyle{plain}

\newtheorem{corollary}{Corollary}

\newtheorem{definition}{Definition}

\newtheorem{lemma}{Lemma}

\newtheorem{proposition}{Proposition}
\newtheorem{remark}{Remark}

\numberwithin{equation}{section}
\input{tcilatex}

\begin{document}
\title[Some general fixed-point theorems for nonlinear mappings]{Some
general fixed-point theorems for nonlinear mappings connected with one
Cauchy theorem }
\author{Kamal N. Soltanov}
\address{National Academy of Sciences of Azerbaijan, Baku, Azerbaijan}
\email{sultan\_kamal@hotmail.com }
\urladdr{https://www.researchgate.net/profile/Kamal-Soltanov/research\ }
\date{}

\begin{abstract}
In this work, using a new geometrical approach we study to the existence of
the fixed-point of mappings that independence of the smoothness, and also of
their single- values or multi values. This work proved the theorems that
generalize in some sense the Brouwer and Schauder fixed-point theorems, and
also such type results in multi-valued cases. One can reckon this approach
is based on the generalization of the one theorem Cauchy and on the
convexity properties of sets. As the used approach is based on the geometry
of the image of the examined mappings that are independent of the
topological properties of the space we could to prove the general results
for almost every vector space. The general results we applied to the study
of the nonlinear equations and inclusions in VTS, and also by applying these
results are investigated different concrete nonlinear problems. Here
provided also sufficient conditions under which the conditions of the
theorems will fulfill.
\end{abstract}

\maketitle

\section{\protect\bigskip Introduction}

The aim of this work is to show the existence of the fixed-point of the
mapping one can study without usual conditions such as some smoothness of
the examined mapping and also the compactness of the subset into which acts
the examined mapping. Namely, we show the existence of the fixed-point of
the nonlinear mapping (in single-valued and multivalued cases) one can study
by using certain geometrical conditions and without the above-mentioned
conditions. This approach allows studying nonlinear equations and inclusions
with the mappings $f$ acting from one vector topological space ($VTS$) $X$
to another $VTS$ $Y$.

In the $1-$dimension case, one can assume to this question was studied by
Cauchy proving the theorem on the "means" value of the continuous functions,
that after is generalized by Hadamard to the functions that act between of $%
1-$dimension vector spaces and with the condition the image of some given
connected subset is the connected subset. Later can be to see, how need to
formulate the mentioned theorem for this remark. In particular, these
results answer the question of the existence of the fixed point. Later
Poincare proved in the $2-$dimension case the existence of the fixed-point
for some continuous mapping under sufficiently severe constraints and in
this article [39], he proposed to prove a result that can show when the
continuous mappings can have a fixed-point, moreover result of such type is
very important. By the question posed Poincare many mathematics began to
study the existence of the fixed point of the mappings. This problem was
solved in various variants by Brouwer, who was proved the well-known Brouwer
fixed-point theorem for the finite-dimensional cases, and by Banach, who was
proved the well-known Banach fixed-point theorem for the contractions
operator on the Banach spaces. For brevity, we won't cite other results
relative to this problem (for more see, e.g. [2-11, 13, 18, 28, 32, 35, 38,
], etc.). The problem of the existence of the fixed-point of mappings in the
infinite-dimensional case was investigated in many works (see, e.g. [2-10,
20, 28, 31, 32, 36, 39] and their references), where were used various
approaches. Schauder generalized the Brouwer theorem to the
infinite-dimensional case and later this result were generalized in the
joint work by Leray and Schauder. There exist some generalization of the
Banach theorem in the case when the operator is nonexpansive. We should be
noted the fixed-point theorems of Schauder and Fan-Kakutani were had certain
generalizations, but in this work are obtained the generalizations of these
and also the fixed-point theorems of Brouwer, Kakutani in the other sense.

This work is proved the theorems that generalize in some sense of the Cauchy
theorem (and some other results) to the finite-dimensional and the
infinite-dimensional cases moreover, are proved some new fixed-point
theorems. All obtained here results are based on the geometry of the image
of the examined mappings. The used here approach connected, in some sense,
with the topological method, more exactly with the Brouwer and Schauder
theorems. In this work is generalized also the lemma that was called
"acute-angled lemma" that is the variant of the Brouwer theorem. It should
be noted the "acute-angled lemma" together with Galerkin methods
successfully is applied under investigations of the nonlinear partial
differential equations and inequations (see, e.g. [14, 15, 30, 32, 36, 38,
40, 42], etc.).

The obtained general results are applied to the study of the solvability of
the nonlinear equations and inclusions in VTS, in particular, to the
investigations of the boundary value problems for nonlinear equations and
inclusions. Here provided also sufficient conditions, which show when the
conditions of the theorems are fulfilled.

So, here considered a nonlinear mapping $f:D\left( f\right) \subseteq
X\longrightarrow Y$, where $X$ , $Y$ are the $VTS$ and is investigated the
question: under which conditions a given $y\in Y$ belongs to the image $%
f\left( G\right) $ of some subset $G\subseteq D\left( f\right) $? It is
clear this question is equivalent to the question on the solvability of the
equation $f\left( x\right) =y$, and also the inclusion $f\left( x\right) \ni
y$ depending on the single-valued or multi-valued of mapping $f$. For the
study, the question has used the approach based on the geometrical structure
of the image $f\left( G\right) $ of the given subset $G\subseteq D\left(
f\right) $ that apriori isn't connected with any smoothness of the mapping $%
f $. Therefore, this approach one can call a geometrical approach.

Now we will lead the simple variant of the main fixed-point theorem of this
work on the Hilbert space.

\begin{theorem}
(Fixed-point theorem) Let $X$ be a Hilbert space, $B_{r}^{X}\left( 0\right)
\subset X$ is the closed ball and the mapping $f$ acting in $X$ be such that 
$f\left( B_{r}^{X}\left( 0\right) \right) \subseteq B_{r}^{X}\left( 0\right) 
$. Then if the image $f_{1}\left( B_{r}^{X}\left( 0\right) \right) \subset X$
be a open (or closed) convex set then there exists such $x_{0}\in
B_{r}^{X}\left( 0\right) $ that $f\left( x_{0}\right) =x_{0}$ (or $f\left(
x_{0}\right) \ni x_{0}$ if \ $f$ $\ $is the multi-valued), where $%
f_{1}\left( x\right) \equiv x-f\left( x\right) $ for $\forall x\in
B_{r}^{X}\left( 0\right) $.
\end{theorem}

In this article, the above-posed questions in multi-dimensional (Section \ref%
{Sec_1}) and infinite-dimensional (Section \ref{Sec_2}) cases are
investigated. The case of the reflexive Banach spaces is studied separately
in Section \ref{Sec_3}. Section \ref{Sec_4} the nonlinear equations and
inclusions in Banach spaces by application of the obtained general results
are investigated. Section \ref{Sec_5} some sufficient conditions for
fulfillment the conditions of the theorems are obtained, and Section \ref%
{Sec_6} leads the concrete examples of problems.

\section{\label{Sec_1}Some Generalization of the Brouwer Fixed-Point Theorem
and its implications}

In the beginning we will provide some known results that are necessary for
next (see, [8, 9, 13, 16, 17, 19, 21-27, 35]).

We denote by $B_{r}^{R^{n}}\left( x_{0}\right) $ the closed ball in $R^{n}$, 
$n\geq 1$ and by $S_{r}^{R^{n}}\left( x_{0}\right) $ the sphere (boundary of
the ball $B_{r}^{R^{n}}\left( x_{0}\right) $) with a center $x_{0}\in R^{n}$
and radius $r>0$.

\begin{theorem}
\label{Th_1}(see, [3, 15]) Let $f$ acted in $R^{n}$ and for some $r>0$ on
the closed ball $B_{r}^{R^{n}}\left( 0\right) \subseteq D\left( f\right) $
satisfies conditions: 1) $f$ is continuous; 2) the inequality $\left\langle
f\left( x\right) ,x\right\rangle \geq 0$ for any $x\in S_{r}^{R^{n}}\left(
0\right) $. Then there exists, at least, one element $x_{1}\in
B_{r}^{R^{n}}\left( 0\right) $ such that $f\left( x_{1}\right) =0$. (For the
proof see, [3, 15].)
\end{theorem}

Well-known that the closed ball $B_{r}^{R^{n}}\left( 0\right) $ one can
change with the closed convex absorbing subset that is homeomorphic to the
ball. Below will be shown that this theorem is the generalization to the
finite-dimension case of one theorem Cauchy on continuous functions acted in 
$R^{1}$, moreover, this theorem is one of the answers to the posed question
in the case when mapping is continuous, and the space is $R^{n}$.

\begin{theorem}
\label{Th_2} (see, [21]) Any two nonintersecting nonempty convex sets of
linear space can separate if either one of these has a nonempty interior or
they be subsets of finite-dimension space.
\end{theorem}

\begin{theorem}
\label{Th_3} (see, [22]) If $M$ is the closed convex subset of the locally
convex linear topological space $X$ and $x_{0}\notin M$ then there exists a
nonzero linear continuous functional $x_{0}^{\ast }\in X^{\ast }$ that
separates $M$ and $x_{0}$, i.e. will found constants $c>0$, $c_{0}>0$ such
that 
\begin{equation}
\sup \left\{ \left. \left\langle x,x_{0}^{\ast }\right\rangle \right\vert \
x\in M\right\} \leq c_{0}-c<c_{0}=\left\langle x_{0},x_{0}^{\ast
}\right\rangle ,  \label{2.1}
\end{equation}%
where $\left\langle \circ ,\circ \right\rangle $ is the dual form for the
pair $\left( X,X^{\ast }\right) $.
\end{theorem}

\begin{theorem}
\label{Th_4} (see, [35]) Let $X$ be a real vector topological space, $M$ be
an open convex subset in $X$ and $N$ be a convex subset in $X$, and also $%
M\cap N=\varnothing $. Then there exist such linear continuous functional $%
x_{0}^{\ast }$ on $X$ and real number $\alpha \in R^{1}$ that $\left\langle
x,x_{0}^{\ast }\right\rangle \geq \alpha $ for $\forall x\in M$ and $%
\left\langle x,x_{0}^{\ast }\right\rangle <\alpha $ for $\forall x\in N$.
\end{theorem}

As, below will use also some concepts from the theory of vector spaces (see,
[5, 7, 12, 13, 25, 35, ]) therefore here we will reduce these. Let $X$ be a
vector space, $L\subset X$ called a linear manifold in $X$ (or an affine
subspace of $X$), if $L$ is a certain shift of some subspace $X_{0}$ of $X$,
i.e. there exists $x_{0}\in X$ such that $L=X_{0}+x_{0}$; $L\subset X$
called a hyperplane in $X$ if $L$ is a maximal affine subspace of $X$ that
is different of $X$: $L\equiv \left\{ \left. x\in X\right\vert \
\left\langle x,x_{0}^{\ast }\right\rangle =\alpha \right\} $ for some $%
\alpha \in R^{1}$ and some nonzero linear continuous functional $x_{0}^{\ast
}\in X^{\ast }$. If the convex set $M\subset X$ is the subset of affine
subspace generated over $M$ then the totality of the interior of the set $M$
is called relatively interior elements of $M$ from $X$ and denoted by $%
\mathit{ri}M$ . A convex set $K$ from the vector space $X$ is called a
convex cone with a vertex on zero of $X$ if $K$ invariant relative to all
homothety, i.e. $x\in K\Longrightarrow \alpha x\in K$, for $\forall \alpha
\in R_{+}^{1}$, if $0\in K\subset X$ then $K$ is called a pointed cone.

Let $R^{n}$ ($n\geq 1$) $n-$dimension Euclid space, $f$ be nonlinear mapping
acting in $R^{n}$, $B_{r}^{R^{n}}\left( x_{0}\right) \subset R^{n}$ be a
closed ball with a center $x_{0}$ and a radius $r>0$ and $%
S_{r}^{R^{n}}\left( x_{0}\right) $ be its boundary, denote by $D\left(
f\right) \subseteq $ $R^{n}$ the domain of $f$.

\begin{theorem}
\label{Th_5} Let the subset $G\subset R^{n}$ belong to $D\left( f\right) $
and the following conditions are fulfilled: 1) $f\left( G\right) $ be a
convex subset in $R^{n}$; 2) there exists a subspace $X$ of $R^{n}$ with the
dimension $0<k\leq n$ such that for any $x_{0}\in S_{1}^{R^{n}}\left(
0\right) \cap X$ there exists such $x_{1}\in G\cap X$ that 
\begin{equation}
\left\{ \left\langle y,x_{0}\right\rangle \left\vert \ y\in f\left(
x_{1}\right) \cap X\right. \right\} \cap R_{+}^{1}\neq \varnothing ,\quad
R_{+}^{1}=\left( 0,\infty \right) ,  \label{2.2}
\end{equation}%
holds (here the $R_{+}^{1}$ can be substituted by $R_{-}^{1}$). Then $0\in
f\left( G\right) $, i.e. $\exists \widehat{x}\in G$ that $0\in f\left( 
\widehat{x}\right) $ (if $f$ single-valued then $f\left( \widehat{x}\right)
=0$).
\end{theorem}

\begin{remark}
\label{R_1}We will formulate this result for $R^{1}$ in the following form.
Let the mapping $f$ acting in $R^{1}$and the image $f\left( G\right) $ of
some bounded subset $G\subset R^{1}$ is the connected subset of $R^{1}$ then
a point $C\in R^{1}$ belong to $f\left( G\right) $ and consequently there is
point $c\in G$ such that $C\in f\left( c\right) $ if there exist such points 
$a,b\in G$ ($a<b$) that the inequations $\left( f\left( a\right) -C\right)
\cdot \left( -1\right) \geq 0$ and $\left( f\left( b\right) -C\right) \cdot
\left( 1\right) \geq 0$ are fulfilled. The proof follows from the
connectivity of $f\left( G\right) $.

This result is generalized of the results Cauchy and Hadamard as
one-dimension vector space is equivalent to $R^{1}$.
\end{remark}

Before the proof of Theorem \ref{Th_5}, we will prove some particular
variants of this, which have independent interest. In beginning, we bring a
simple variant.

\begin{lemma}
\label{L_1} Let for some $r>0$ the image $f\left( B_{r}^{R^{n}}\left(
0\right) \right) $ of the ball $B_{r}^{R^{n}}\left( 0\right) $ is closed (or
opened) convex set and is fulfilled the inequation $\left\{ \left\langle
y,x\right\rangle \left\vert \ y\in f\left( x\right) \right. \right\} \cap
\left( 0,\infty \right) \neq \varnothing $ for $\forall x\in
S_{r}^{R^{n}}\left( 0\right) $ then $0\in f\left( B_{r}^{R^{n}}\left(
0\right) \right) $, i.e. $\exists x_{1}\in B_{r}^{R^{n}}\left( 0\right) $
such that $0\in f\left( x_{1}\right) $. (if $f$ single-valued then $f\left(
x_{1}\right) =0$).
\end{lemma}

\begin{proof}
The proof we bring from inverse. Let $0\notin f\left( B_{r}^{R^{n}}\left(
0\right) \right) \equiv M$. According to condition $M$ \ is closed or opened
convex set in $R^{n}$ then due to the separation theorem of convex subsets
there exists a linear bounded functional $\overline{x}\in R^{n}$ separating
of $M$ and zero under the conditions of the lemma. Since $%
B_{r}^{R^{n}}\left( 0\right) $ is the absorbing set of $R^{n}$, therefore
one can assume that functional $\overline{x}$ belongs to $%
S_{r}^{R^{n}}\left( 0\right) $. Whence we get $\left\langle y,\overline{x}%
\right\rangle <0$ for $\forall y\in f\left( \overline{x}\right) $\ that
contradicts the condition of the lemma. Lemma proved.
\end{proof}

Analogously is proved the following result.

\begin{lemma}
\label{L_2} Let $f\left( B_{r}^{R^{n}}\left( 0\right) \right) $ is the covex
set and on the sphere $S_{r}^{R^{n}}\left( 0\right) $ is fulfilled the
inequation $\left\{ \left\langle y,x\right\rangle \left\vert \ y\in f\left(
x\right) \right. \right\} \cap \left( 0,\infty \right) \neq \varnothing $
for $\forall x\in S_{r}^{R^{n}}\left( 0\right) $ then $0\in f\left(
B_{r}^{R^{n}}\left( 0\right) \right) $.
\end{lemma}

\begin{lemma}
\label{L_3} Let $f\left( B_{r}^{R^{n}}\left( x_{0}\right) \right) $ is the
covex set and there exists mapping $g$ acting in $R^{n}$ such that $g\left(
S_{r}^{R^{n}}\left( x_{0}\right) \right) $ is the boundary of an absorbing
subset of $R^{n}$. Then if for $\forall x\in S_{r}^{R^{n}}\left(
x_{0}\right) $ the expression 
\begin{equation}
\left\{ \left\langle y,z\right\rangle \left\vert \ \forall y\in f\left(
x\right) ,\ \forall z\in g\left( x\right) \right. \right\} \cap \left(
0,\infty \right) \neq \varnothing  \label{2.3}
\end{equation}%
holds then $0\in f\left( B_{r}^{R^{n}}\left( x_{0}\right) \right) $.
\end{lemma}

\begin{remark}
\label{R_2}It is clear that if $f$ single-valued then expression (\ref{2.3})
one can rewrite as 
\begin{equation*}
\left\{ \left\langle f\left( x\right) ,z\right\rangle \left\vert \ \ \forall
z\in g\left( x\right) \right. \right\} \cap \left( 0,\infty \right) \neq
\varnothing .
\end{equation*}
\end{remark}

\begin{proof}
(Lemma \ref{L_3}) As above the proof we bring from inverse. Let $0\notin
f\left( B_{r}^{R^{n}}\left( x_{0}\right) \right) $ then repeating of
previous argue we get there is such point (a linear bounded functional) $%
z_{0}\in S_{1}^{R^{n}}\left( 0\right) \subset R^{n}$ that $\left\langle
y,z_{0}\right\rangle \leq 0$ for $\forall y\in f\left( x\right) $ and $%
\forall x\in B_{r}^{R^{n}}\left( x_{0}\right) $ due to the condition on
convexity of $f\left( B_{r}^{R^{n}}\left( x_{0}\right) \right) $ and Theorem %
\ref{Th_2} on separation of the convex sets. But this contradicts the
condition of the lemma, consequently, Lemma proved.
\end{proof}

\begin{corollary}
\label{C_1} Let the mapping $f$ acting in $R^{n}$ such that for an subspace $%
X$ dimension $k\leq n$ the $f\left( B_{r}^{R^{n}}\left( x_{0}\right) \right)
\cap X$ is convex set and there exists mapping $g$ acting in $R^{n}$ such
that $g\left( S_{r}^{R^{n}}\left( x_{0}\right) \cap X\right) $ is the
boundary of an absorbing subset of $X$. Then if for $\forall x\in
S_{r}^{R^{n}}\left( x_{0}\right) \cap X$ the expression (\ref{2.3}) is
fulfilled then $0\in f\left( B_{r}^{R^{n}}\left( x_{0}\right) \right) $.
\end{corollary}

The proof follows from the proof of the Lemma \ref{L_3}, more exacly, the
above reasoning enough to conduct for the convex set $f\left(
B_{r}^{R^{n}}\left( x_{0}\right) \right) \cap X$ (that is subset of $f\left(
B_{r}^{R^{n}}\left( x_{0}\right) \right) $) in the subspace $X$, since $X$
is also the $k-$dimensional space.

\begin{proof}
(of Theorem \ref{Th_5}) As to see from the proofs the above Lemmas the
selection of the subset from the domain of the examined functions isn't
essential but essentially the convexity of its image. So, if $X$ is $R^{n}$
then the proof follows from Lemmas \ref{L_1} and \ref{L_2}. Therefore, let $%
X=$ $R^{k}$, $1\leq k<n$. Assume $0\notin f\left( G\right) $ and the affine
space generated over the convex set $f\left( G\right) $ is the hyperplane $%
L\subset R^{n}$ and show that under condition 2 of the theorem $L$ couldn't
be the hyperplane different of the subspace of $R^{n}$. Assume $0\notin L$,
i.e. $L$ isn't a subspace then there exist $x_{0}\in $ $f\left( G\right) $
and subspace $X_{0}$, $\dim X_{0}<n$ such that $L=X_{0}+x_{0}$.
Consequently, $X_{0}$ is the subspace generated over $f\left( G\right)
-x_{0} $. Since $0\notin L$ there is an element (point) $z$ on $%
S_{1}^{R^{n}}\left( 0\right) $ such that $\left\langle y,z\right\rangle <0$
for $\forall y\in L$, due to the separation theorems for the
finite-dimensions spaces ([see,18]) and, consequently, for $\forall y\in
f\left( G\right) \subset L$ but according to condition 2 of Theorem \ref%
{Th_5} there is such point $y_{0}\in f\left( G\right) $ that $\left\langle
y_{0},z\right\rangle >0$. Whence, $L$ is the subspace $X_{0}$ of $R^{n}$, i.
e. $X_{0}\subseteq X\subset R^{n}$. On the other hand, condition 2 satisfies
for $\forall x\in $ $S_{1}^{R^{k}}\left( 0\right) $, therefore, $X_{0}=X$,
i. e. $X$ is the $k-$dimension subspace of $R^{n}$. Thus we arrive at the
case considered in Lemma \ref{L_3}. Then using the result of this lemma we
get $0\in f\left( G\right) $.
\end{proof}

From Theorem \ref{Th_5} follows the correcness of the following result.

\begin{corollary}
\label{C_2} Let $f$ acting in $R^{n}$ and there is such subset $G$ in the
image $\func{Im}f$ \ that satisfy the following conditions:

i) $G$ is the convex set; ii) There exists such subspace $X\subseteq R^{n}$
with the dimension $k:1\leq k\leq n$ that for $\forall z\in $ $%
S_{1}^{R^{n}}\left( 0\right) \cap X$ there exists $y\in G$ such that $%
\left\langle z,y\right\rangle >0$ or ($\left\langle z,y\right\rangle <0$).
Then $0\in \func{Im}f$ .
\end{corollary}

It is clear that one can the above inequation rewrite in the form 
\begin{equation}
\left\{ \left\langle z,y\right\rangle \left\vert \ \forall z\in
S_{1}^{R^{n}}\left( 0\right) \cap X,\ \exists x\in D\left( f\right) ,\text{
for some }y\in G\cap f\left( x\right) \right. \right\} \cap \left( 0,\infty
\right) \neq \varnothing .  \label{2.4}
\end{equation}

\begin{theorem}
\label{Th_6} Let $f$ acting in $R^{n}$ and the ball $B_{r}^{R^{n}}\left(
x_{0}\right) $ with center on $x_{0}$ and the radius $r>0$ belong the domain
of $f$. Let $f\left( B_{r}^{R^{n}}\left( x_{0}\right) \right) \subseteq
B_{r}^{R^{n}}\left( x_{0}\right) $ and for some subspace $R^{k}$, $k:1\leq
k\leq n$, takes place $f\left( B_{r}^{R^{n}}\left( x_{0}\right) \cap
R^{k}\right) \subset B_{r}^{R^{n}}\left( x_{0}\right) \cap R^{k}$, and let $%
f_{1}$is the operator in the form $f_{1}\left( x\right) \equiv Ix-f\left(
x\right) $ for $\forall x\in B_{r}^{R^{n}}\left( x_{0}\right) $. Assume $%
f_{1}\left( B_{r}^{R^{n}}\left( x_{0}\right) \right) $ is the convex subset
then the mapping $f$ has a fixed point in $B_{r}^{R^{n}}\left( x_{0}\right) $%
.
\end{theorem}

The proof immediately follows from the above-mentioned results, therefore
here it doesn't will be provide. Since, according to conditions of Theorem %
\ref{Th_6} imply the fulfillment the conditions of Theorem \ref{Th_5} for
the mapping $f_{1}$ on the ball $B_{r}^{R^{n}}\left( x_{0}\right) $.

\begin{remark}
\label{R_3} Theorem \ref{Th_6} is the Fixed-point Theorem, where doesn't
condition onto mapping $f$ that usually assumed in theorems of such type
(such as its smoothness, single-value, or multi-value). Consequently, this
theorem can consider as the generalization of such type theorems in the
above sense.

Now we provide an example that shows the rigor of inequation in condition 2
of the above results is essential. For simplicity assume $n=2$, i.e. mapping 
$f$ acting in $R^{2}$and the image of ball $B_{r}^{R^{2}}\left( 0\right)
\subset D\left( f\right) $ ( $r>0$) is 
\begin{equation}
f\left( B_{r}^{R^{2}}\left( 0\right) \right) =B_{r}^{R^{2}}\left( 0\right)
\cap \left\{ x=\left( x_{1},x_{2}\right) \in R^{2}\left\vert ,\
x_{2}>0\right. \right\} \cup \left[ \left( \frac{r}{2},0\right) ,\left(
r,0\right) \right] .  \label{2.5a}
\end{equation}%
It isn't difficult to see that for $\forall z\in S_{1}^{R^{2}}\left(
0\right) $ there exists an $y\in f\left( B_{r}^{R^{2}}\left( 0\right)
\right) $ such that $\left\langle z,y\right\rangle \geq 0$ but $0\notin
f\left( B_{r}^{R^{2}}\left( 0\right) \right) $. In particular, such mapping $%
f:B_{r}^{R^{2}}\left( 0\right) \longrightarrow R^{2}$ one can define e.g. by
following way 
\begin{equation}
f\left( x_{1},x_{2}\right) =\left\{ 
\begin{array}{c}
\left( y_{1},y_{2}\right) =\left( x_{1},-x_{2}\right) ,\quad -r\leq
x_{1}\leq r,\ -r\leq x_{2}<0 \\ 
\left( y_{1},y_{2}\right) =\left( x_{1},x_{2}\right) ,\quad -r\leq x_{1}\leq
r,\ 0<x_{2}\leq r\  \\ 
\left( y_{1},y_{2}\right) =\left( \frac{x_{1}}{2}+r,0\right) ,\quad -r\leq
x_{1}\leq 0,\ x_{2}=0\  \\ 
\left( y_{1},y_{2}\right) =\left( \frac{r}{2}+x_{1},0\right) ,\quad 0<x_{1}<%
\frac{r}{2},\ x_{2}=0\  \\ 
\left( y_{1},y_{2}\right) =\left( x_{1},0\right) ,\quad \frac{r}{2}\leq
x_{1}\leq r,\ x_{2}=0%
\end{array}%
\right. .  \label{2.5}
\end{equation}

The essentialness of the condition that image $f\left( G\right) $, (for $%
G\subseteq D\left( f\right) $) of the examined mapping is the convex set is
obviously. Indeed, if assume that in the above example we have $f\left(
B_{r}^{R^{2}}\left( 0\right) \right) =B_{r}^{R^{2}}\left( 0\right)
\backslash B_{r/2}^{R^{2}}\left( 0\right) $ and $f\left( 0\right) =k$ then
condition 2 satisfies but $B_{r}^{R^{2}}\left( 0\right) \backslash
B_{r/2}^{R^{2}}\left( 0\right) \cup \left\{ k\right\} $ isn't a convex set
and obviously $0\notin B_{r}^{R^{2}}\left( 0\right) \backslash
B_{r/2}^{R^{2}}\left( 0\right) \cup \left\{ k\right\} $.
\end{remark}

It is clear that Lemma \ref{L_1} is the generalization of the "acute-angle"
lemma (see, e.g. [3, 12, 15, 16, 40, 42) in the above-mentioned sense.
Consequently, it is also the generalization of the Brauwer fixed-point
theorem (see, e.g. [3, 6, 15, 32, 34]) since it is equivalent to
"acute-angle" lemma, and also the Kakutani fixed-point theorem (see, e.g.
[13, 16, 19, 26]) for the multi-valued case.

\section{\label{Sec_2}Generalization of Theorem \protect\ref{Th_5} to
Infinite-Dimension Cases and their Corollaries}

In this section, we will generalize the results of the previous section to
the general spaces. Such generalization is possible since the convexity
concept is independent of the dimension and the topology of the space. In
the beginning, we will generalize results to the case of the Hausdorff $VTS$%
. Let $f:D\left( f\right) \subseteq X\longrightarrow Y$ be some mapping in
general nonlinear.

Let $X$ and $Y$ be locally convex $VTS$ ($LVTS$) and $X^{\ast }$ and $%
Y^{\ast }$ be their dual spaces, respectively. Denote by $\partial U$ ($%
\partial U^{\ast }$) the boundary of a convex closed bounded absorbing set $%
U $ ($U^{\ast }$) in the appropriate space (see, [33])\footnote{%
This is chosen for simplicity. In reality, here is sufficient to choose this
set as the boundary of closed balanced absorbing set that will be seen from
the further discuss.}. In particular, any ball $B_{r}^{X}\left( 0\right) $
and any vicinity of $0\in X$ with the nonempty interior is an absorbing set
in $X$, therefore the sphere $S_{r}^{X}\left( 0\right) $ is the set of the
type $\partial U$.

\begin{theorem}
\label{Th_7} Let $f:D\left( f\right) \subseteq X\longrightarrow Y$ be some
mapping and $G\subseteq D\left( f\right) $ be such set that for the image $%
f\left( G\right) $ one of the following conditions satisfies:

i) $f\left( G\right) $ is convex set with nonempty interior in $Y$; ii) $%
f\left( G\right) $ is open (closed) convex set in $Y$. Then if for any
linear continuous functional $y^{\ast }\in \partial U^{\ast }\subset Y^{\ast
}$ there exists $x\in G$ such that the following inequation

$\left\{ \left\langle y^{\ast },y\right\rangle \left\vert \ y\in f\left(
x\right) \right. \right\} \cap \left( 0,\infty \right) \neq \varnothing $ in
the case i); $\left\{ \left\langle y^{\ast },y\right\rangle \left\vert \
y\in f\left( x\right) \right. \right\} \cap \left[ 0,\infty \right) \neq
\varnothing $ in the case ii) holds then $0\in f\left( G\right) $.
\end{theorem}

\begin{proof}
The proof we bring from inverse. Let $0\notin f\left( G\right) $ then there
exists such linear continuous functional $y_{0}^{\ast }\in \partial U^{\ast
} $ that $\left\langle y_{0}^{\ast },y\right\rangle \leq 0$ for any $y\in
f\left( G\right) $ in the case \textit{i),} due to convexity of $f\left(
G\right) $ according to the separation theorems of the convex sets in the $%
LVTS$ (see, e.g. [15, 16, 17, 19]). By the analogous reasoning to the
previously, we get that in case \textit{ii)} according to the separation
theorems of the convex sets in the $LVTS$ there exists such linear
continuous functional $y_{0}^{\ast }\in \partial U^{\ast }$ that $%
\left\langle y_{0}^{\ast },y\right\rangle <0$ for any $y\in f\left( G\right) 
$. Later on, by using the reasoning lead in the previous section, we obtain
the correctness $0\in f\left( G\right) $ due to the obtained contradiction
with the conditions of the theorem.
\end{proof}

It should be noted that such type result is correct and also for the
Hausdorff $VTS$.

Consider the case when $X$ and $Y$ be vector spaces ($VS$) and $f:D\left(
f\right) \subseteq X\longrightarrow Y$ be some mapping.

\begin{theorem}
\label{Th_8} Let mapping $f:D\left( f\right) \subseteq X\longrightarrow Y$
and for some $G\subseteq D\left( f\right) $ the image $f\left( G\right) $ be
convex subset in $Y$. Assume there exist such subspace $Y_{1}$ that for any
subspace $Y_{0}$ with $co\dim _{Y_{1}}Y_{0}=1$ the following expressions
fulfill 
\begin{equation}
f\left( G\right) \cap \left( Y_{1}\right) _{Y_{0}}^{+}\neq \varnothing \ \&\
f\left( G\right) \cap \left( Y_{1}\right) _{Y_{0}}^{-}\neq \varnothing .
\label{3.1}
\end{equation}

Then zero of the space $Y$ belongs to $f\left( G\right) \subseteq Y$, i.e.
there exists such $x_{0}\in G$ that $0\in f\left( x_{0}\right) $.
\end{theorem}

\begin{proof}
Let $0\notin f\left( G\right) $. \ According to the convexity of $f\left(
G\right) $, there exist the subspace $Y_{2}\subseteq Y_{1}$ relative to
which the set $f\left( G\right) $ is the convex set with nonempty $C-$%
interior (see, [22]) moreover, sufficiently choose the affine space
generated over $f\left( G\right) $, as due to conditions of Theorem on the $%
f\left( G\right) $ it will be the subspace of $Y$.\footnote{%
Brief notes on some properties of the linear spaces. Let $Y$ be $VS$. Each
hyperplane $L$ of $Y$ is equivalent to a subspace $Y_{0}$ of $Y$ with $%
co\dim _{Y}Y_{0}=1$ it separates the space to two half-space, which can be
denoted as $Y_{L}^{+}$ and $Y_{L}^{-}$. The $C-$interior point is defined as
ensuing way: a point $y_{0}$ of the subset $U\subset $ $Y$ called $C-$%
interior point if for $\forall y\in U$ there exists $\varepsilon >0$ such
that for all $\delta :0<\left\vert \delta \right\vert <$ $\varepsilon $ the
inclusion $y_{0}+\delta y\in U$ holds.}

Consequently, without loss of generality can be account that $Y_{2}\equiv
Y_{1}$. For simplicity, in the beginning, assume $Y_{1}=Y$. Then we obtain
there exists such subspace $Y_{3}\subset Y$ with respect to which $f\left(
G\right) $ belongs to one of half-spaces $Y_{Y_{3}}^{+}$ or $Y_{Y_{3}}^{-}$
(see, [23, 24]) and $f\left( G\right) \cap Y_{3}=\varnothing $, according to
the separation theorem of the convex subsets in $VS$ (see, [21]). But this
contradicts the condition (\ref{3.1}) this shows the correctness of the
state of the theorem in the case when $Y_{1}=Y$.

Let now $Y_{1}\subset Y$ and denote the mapping $f_{0}\left( x\right) =$ $%
f\left( x\right) \cap Y_{1}$ for $\forall x\in G$ then $f_{0}\left( G\right)
=$ $f\left( G\right) \cap Y_{1}$. Clear that $f_{0}\left( G\right) $ is
convex set in $Y_{1}$. Consequently, one can repeat the above reasoning for
the mapping $f_{0}\left( x\right) $ and the space $Y_{1}$ as independent $VS$
from $Y$ that again will give the same result as previous, which contradicts
the condition (\ref{3.1}).

Thus we obtain the correctness of the state of the Theorem \ref{Th_8}.
\end{proof}

From above-mentioned theorems ensue the correctness of the following
fixed-point theorem.

\begin{corollary}
(Fixed-Point Theorem)\label{C_3} Let the mapping $f$ acts into the space $X$
that is (a) $VS$ or (b) $LVTS$ and the convex subset $G\subseteq D\left(
f\right) $. Let $f:G\longrightarrow G$ and denote by $f_{1}:G\longrightarrow
G$ the mapping $f_{1}=Id-f$ ($f_{1}\left( x\right) =Id\ x-f\left( x\right) $
for $\forall x\in G$). Assume the mapping $f_{1}$ on the set $G$ satisfies
condition (\ref{3.1}), in the case (a);the condition of the Theorem \ref%
{Th_7}, in the case (b). Then the mapping $f$ in the set $G$ has a
fixed-point, i.e. there exists $x_{0}\in G$ such that $x_{0}\in f\left(
x_{0}\right) $ (if $f$ is single-value mapping then $f\left( x_{0}\right)
=x_{0}$).
\end{corollary}

The proof is obvious. If this result compares with the Schauder and
Fan-Kakutani fixed-point theorems (see, e.g. [13, 17, 22, 23, 39]) then can
be to see it generalized these in above-mentioned sense.

It isn't difficult to see the correctness of following result.

\begin{corollary}
\label{C_4} Let the mapping $f$ acting from $LVTS$ $X$ to $LVTS$ $Y$ on some
subset $G\subseteq D\left( f\right) $ satisfies the following condition:
there exists a subspace $Y_{0}$ of $Y$ ($Y_{0}\subseteq Y$) such that $%
f\left( G\right) \cap Y_{0}$ is a convex set, moreover either (a) is open
(or closed), or with the nonempty interior with respect to $Y_{0}$. Then if
for each $y^{\ast }\in \partial U^{\ast }\cap Y_{0}^{\ast }\subseteq Y^{\ast
}$ there exists such element $x\in G$ that the expression 
\begin{equation*}
(a)\ \left\{ \left\langle y,y^{\ast }\right\rangle \left\vert \ y\in f\left(
x\right) \cap Y_{0}\right. \right\} \cap \left[ 0,\infty \right) \neq
\varnothing ;
\end{equation*}%
\begin{equation*}
\left( b\right) \ \left\{ \left\langle y,y^{\ast }\right\rangle \left\vert \
y\in f\left( x\right) \cap Y_{0}\right. \right\} \cap \left( 0,\infty
\right) \neq \varnothing
\end{equation*}%
holds. Then $0\in Y$ belongs to $f\left( G\right) $, i.e. $0\in f\left(
G\right) $.
\end{corollary}

Now we will reduce examples showing the essentialness of the conditions of
the above-proved theorems. Let $X$ be a reflexive Banach space and $%
Y=X^{\ast }$, i.e. dual space to the $X$. Assume $f:D\left( f\right)
\subseteq X\longrightarrow X^{\ast }$ and $B_{r_{0}}^{X}\left( x_{0}\right)
\subseteq D\left( f\right) $, moreover $f\left( B_{r_{0}}^{X}\left(
x_{0}\right) \right) =B_{r_{0}}^{X^{\ast }}\left( x_{0}^{\ast }\right)
\subset X^{\ast }$, where $r_{0}>0$ and the centers $x_{0}\in X,x_{0}^{\ast
}\in X^{\ast }$ of these balls such that $\left\Vert x_{0}\right\Vert
_{X},\left\Vert x_{0}^{\ast }\right\Vert _{X^{\ast }}>r_{0}$. Let the
mapping $f$ is such as the duality mapping between of the dula spaces, more
exactly, for $\forall x\in B_{r_{0}}^{X}\left( x_{0}\right) $ fulfill the
ensue expressions $f\left( x\right) =x^{\ast }\in B_{r_{0}}^{X^{\ast
}}\left( x_{0}^{\ast }\right) $ and $\left\langle f\left( x\right)
,x\right\rangle =\left\langle x^{\ast },x\right\rangle =\left\Vert
x\right\Vert _{X}\cdot \left\Vert x^{\ast }\right\Vert _{X^{\ast }}>0$.
Clearly, condition 1 satisfies, but condition 2 doesn't be satisfied as
there exist such $\widehat{x}\in S_{1}^{X}\left( 0\right) $ for which
doesn't be exists $\widetilde{x}\in B_{r_{0}}^{X}\left( x_{0}\right) $
satisfying $\left\langle f\left( \widetilde{x}\right) ,\widehat{x}%
\right\rangle \geq 0$, consequently, the claims of the theorems don't take
place. The essentialness of the convexity of the image is obvious.

We will prove one result in the case when the space $Y$ is $LVTS$, in some
sense, which is sufficient for fulfilling the conditions of Theorem \ref%
{Th_7}, for simplicity, in the case of the single-valued mappings.

\begin{proposition}
\label{Pr_1} Let $X,Y$ be $LVTS$, and $f:D\left( f\right) \subseteq
X\longrightarrow Y$ is the single-valued mapping. Let the image $f\left(
G\right) $ for some subset $G\subseteq D\left( f\right) $ is connected open
or closed body in $Y$. Then for each fixed element $y\in int\ f\left(
G\right) $, there exists such subset of $G$ on which for the mapping $%
f_{1}\left( x\right) =f\left( x\right) -y$ the conditions of Theorem \ref%
{Th_7} are fulfills.
\end{proposition}

\begin{proof}
Indeed, from conditions of proposition imply for each point $y\in $ $f\left(
G\right) $ there exists an open or closed convex neighborhood $V\left(
y\right) \subseteq f\left( G\right) $ contains this point. Then by defining
the preimage $f^{-1}\left( V\left( y\right) \right) \subset G$ as $%
G_{1}\subset G$ it is enough to consider the mapping $f_{1}$ on the subset $%
G_{1}$, i.e. $f_{1}:G_{1}\longrightarrow V\left( y\right) $.
\end{proof}

We won't consider the cases of more general spaces. \ 

\section{\label{Sec_3}On Mappings acted in Reflexive Banach Spaces}

In this section, we will investigate mappings acted from one Banach space to
another. It is clear, the results obtained in Section \ref{Sec_2} remain
correct, and in this case. We will study this case separately since
well-known that the geometry of the reflexive Banach spaces was studied
sufficiently complete that allows proving more exact results. Therefore,
these results can be more applicable for in the detail studying various
problems.

So, in this section whole, we assume $X$ and $Y$ be the reflexive spaces
with the strongly convex norms jointly with their dual spaces (see, e.g. [7,
15, 22, 37, 38], etc.). As well-known, each reflexive Banach space can be
renormalized in such a way this space and its dual will the strongly convex
spaces (see, [37], and also [7, 15]). Consequently, in what follows we will
account that all of the examined reflexive Banach spaces are strongly convex
spaces, without loss of generality.

Let $X$ be strongly convex reflexive Banach space jointly with its dual
space $X^{\ast }$. For simplicity in the beginning assume $Y=X^{\ast }$ and
the mapping $f$ acts from $X$ to $X^{\ast }$. Then the main result of this
section will be formulated as follows.

\begin{theorem}
\label{Th_9} Let $f:D\left( f\right) \subseteq X\longrightarrow X^{\ast }$
be some mapping and $G\subseteq D\left( f\right) $. Assume $f\left( G\right) 
$ be convex subset in $X^{\ast }$ and there exists such subspace $%
X_{0}^{\ast }\subseteq X^{\ast }$ that belongs to the affine space $%
X_{f\left( G\right) }^{\ast }$ generating over $f\left( G\right) $ and
either $co\dim _{X^{\ast }}X_{0}^{\ast }\geq 1$ or $0\in X_{f\left( G\right)
}^{\ast }$; Moreover, there exists such $X_{1}\subseteq X$ that $X_{0}^{\ast
}\subseteq X_{1}^{\ast }$, $co\dim _{X_{1}}X_{0}\geq 0$ and for $\forall
x_{0}\in S_{1}^{X}\left( 0\right) \cap X_{1}$ the inequation 
\begin{equation}
\left\{ \left\langle x^{\ast },x_{0}\right\rangle \left\vert \ \exists x\in
G\ \&\ x^{\ast }\in f\left( x\right) \cap X_{1}^{\ast }\right. \right\} \cap
\left( 0,\infty \right) \neq \varnothing  \label{4.1}
\end{equation}%
holds. Then $0\in f\left( G\right) $, i.e. $\exists x_{1}\in
G\Longrightarrow $ $0\in f\left( x_{1}\right) $ (if the mapping $f$ is
singl-value then $f\left( x_{1}\right) =0$).
\end{theorem}

For the proof of this theorem is necessary some auxiliary results,
therefore, previously need to prove these results. We start with a simple
variant of Theorem \ref{Th_9}.

\begin{lemma}
\label{L_4} Let $f:D\left( f\right) \subseteq X\longrightarrow X^{\ast }$ be
some mapping and $G\subseteq D\left( f\right) $. Assume $f\left( G\right) $
be convex subset in $X^{\ast }$ and there exists such subspace $X_{0}^{\ast
}\subseteq X^{\ast }$ that belongs to the affine space $X_{f\left( G\right)
}^{\ast }$ generated over $f\left( G\right) $ and $co\dim _{X^{\ast
}}X_{0}^{\ast }\geq 1$. Then if for $\forall x_{0}\in S_{1}^{X}\left(
0\right) $ there exists $x\in G$ such that $\left\{ \left\langle f\left(
x\right) ,x_{0}\right\rangle \right\} \cap \left( 0,\infty \right) \neq
\varnothing $ then $0\in f\left( G\right) $.\footnote{%
In what follow we will denote for briefness by $\left\{ \left\langle f\left(
x\right) ,x_{0}\right\rangle \right\} $ the set $\left\{ \left\langle
y,x_{0}\right\rangle \left\vert \ y\in f\left( x\right) \right. \right\} $,
where $\left( y,x_{0}\right) \in \left( X^{\ast },X\right) $.}
\end{lemma}

\begin{proof}
The proof ensues from the Theorem \ref{Th_7} if here the condition for the
case (\textit{i}) of this theorem fulfills. Assume this condition isn't
fulfilled then we will use the following result that will be proved later.
\end{proof}

\begin{lemma}
\label{L_5} Let $f:D\left( f\right) \subseteq X\longrightarrow Y$ be some
mapping and $G\subseteq D\left( f\right) \subseteq X$, where the spaces $X$, 
$Y$ be $LVTS$. Let $U^{\ast }\subset Y^{\ast \text{ }}$be a closed bounded
balanced absorbing set in $Y^{\ast \text{ }}$with the boundary $\partial
U^{\ast }$. Then if for $\forall y^{\ast }\in \partial U^{\ast }$ there
exists $x\in G$ such that $\left\{ \left\langle f\left( x\right) ,y^{\ast
}\right\rangle \right\} \cap \left( 0,\infty \right) \neq \varnothing $
fulfills then the affine space $Y_{f\left( G\right) }$ generated over $%
f\left( G\right) $ , at least, is everywhere dense affine subspace in the
space $Y$.
\end{lemma}

\begin{proof}
Continuation of the proof of Lemma \ref{L_4}. According to Lemma \ref{L_5}
under the condition of Lemma \ref{L_4} \ $X_{f\left( G\right) }^{\ast }$ is,
at least, everywhere dense linear subspace in $X^{\ast }$. Let $X_{1}^{\ast
} $ be a subspace of $X^{\ast }$ that belongs to $X_{f\left( G\right)
}^{\ast } $. Then $f\left( G\right) \cap X_{1}^{\ast }$ is the convex set
with nonempty interiors in $X_{1}^{\ast }$ (see, e.g. [5, 17, 21, 22]). As $%
X $ is the reflexive Banach space then $X_{1}^{\ast }$ also is the reflexive
space. Obviously that under the conditions of Lemma \ref{L_4} on $%
X_{1}^{\ast }$ for $\forall y^{\ast }\in \partial U^{\ast }\cap X_{1}^{\ast
} $ there exists $x\in G$ such that $\left\{ \left\langle f\left( x\right)
\cap X_{1}^{\ast },y^{\ast }\right\rangle \right\} \cap \left( 0,\infty
\right) \neq \varnothing $ fulfills (for the proof see, [32]). Thus, we get
that with respect to $X_{1}^{\ast }$ for the examined mapping all conditions
of the Theorem \ref{Th_7} are fulfilled, therefore using this theorem the
correctness of the claim of Lemma \ref{L_4} is obtained.
\end{proof}

It is remains to prove of the Lemma \ref{L_5}.

\begin{proof}
(of the Lemma \ref{L_5}) The proof we bring from inverse. Let $Y_{f\left(
G\right) }$ isn't the everywhere dense affine subspace in the space $Y$.
Denote the closure of $Y_{f\left( G\right) }$ by $Y_{1}\equiv \overline{%
Y_{f\left( G\right) }}^{Y}$. According to assumption $Y_{1}\subset $ $Y$ and 
$Y_{1}\neq Y$ moreover, which is the closed convex affine subspace. Then
there exists $y^{\ast }\in Y^{\ast }$ such that $\left\langle y^{\ast
},Y_{1}\right\rangle <0$, which is contradicts to the conditin of Lemma \ref%
{L_5}.
\end{proof}

Now we provide one result using that one can to prove the Lemma \ref{L_4} by
another way.

\begin{proposition}
\label{P_1} Let $X$ be strongly convex reflexive Banach space jointly with
its dual space $X^{\ast }$, and $\Omega \subset X$ be bounded convex set.
Then if each of one-dimension subspace $L$ from $X$ intersects $\Omega $,
i.e. $L\cap \Omega \neq \varnothing $ then either $0\in \Omega $ or $\Omega $
is such convex body that $0$ of $X$ belongs to the closure of $\Omega $,
i.e. $0\in \overline{\Omega }$.
\end{proposition}

It is clear that under the condition of this result the affine space
generated over $\Omega $ contains a linear subspace of $X$.

\begin{lemma}
\label{L_6} Let conditions of the Theorem \ref{Th_9} fulfilled, where the
expression (\ref{4.1}) fulfilled in the following form. Let $X_{1}\subset X$
be a subspace and $co\dim _{X_{1}}X_{f\left( G\right) }\geq 0$. Then if for $%
\forall x_{0}\in S_{1}^{X}\left( 0\right) \cap X_{1}$ there exists $x\in G$
such that $\left\{ \left\langle f\left( x\right) \cap X_{1}^{\ast
},x_{0}\right\rangle \right\} \cap \left( 0,\infty \right) \neq \varnothing $
holds. Then $0\in f\left( G\right) $.
\end{lemma}

The proof of this lemma analogously to the proof of the Lemma \ref{L_4},
therefore we don't provide it.

\begin{proof}
(of the Theorem \ref{Th_9}) This proof follows from lemmas \ref{L_4} and \ref%
{L_6}. Indeed Lemma \ref{L_4} shows correctness of the Theorem \ref{Th_9}
from one side, and Lemma \ref{L_6} shows correctness of the Theorem \ref%
{Th_9} from in another side. Consequently, the Theorem \ref{Th_9} complete
proved.
\end{proof}

The correctness of the following statements immediately ensues from the
above-mentioned results.

\begin{corollary}
\label{C_5}(Fixed-Point Theorem) Let $f$ be a mapping acting in the
reflexive Banach space $X$ and $B_{r}^{X}\left( x_{0}\right) \subseteq
D\left( f\right) \subseteq X$ be a closed ball. Assume $f$ map $%
B_{r}^{X}\left( x_{0}\right) $ into itself, where $r>0$ is a number. Let $%
f_{1}$ be a mapping defined as $f_{1}\left( x\right) \equiv x-f\left(
x\right) $ for $\forall x\in B_{r}^{X}\left( x_{0}\right) $. Then if the
image$\ f_{1}\left( B_{r}^{X}\left( x_{0}\right) \right) $ of the ball $%
B_{r}^{X}\left( x_{0}\right) $ is a convex set either open (closed) or$\
f_{1}\left( B_{r}^{X}\left( x_{0}\right) \right) \subset int\
B_{r}^{X}\left( x_{0}\right) $ and $int\ f_{1}\left( B_{r}^{X}\left(
x_{0}\right) \right) \neq \varnothing $ then the mapping $f$ has a fixed
point in the ball $B_{r}^{X}\left( x_{0}\right) $.
\end{corollary}

\begin{corollary}
\label{C_6}(Fixed-Point Theorem) Let $f$ be a mapping acting in the
reflexive Banach space $X$ and $B_{r}^{X}\left( x_{0}\right) \subseteq
D\left( f\right) \subseteq X$ be a closed ball. Assume $f$ map $%
B_{r}^{X}\left( x_{0}\right) $ into itself, where $r>0$ is a number. Let $%
f_{1}$ be a mapping defined as $f_{1}\left( x\right) \equiv x-f\left(
x\right) $ for $\forall x\in B_{r}^{X}\left( x_{0}\right) $. Then if the
image$\ f_{1}\left( B_{r}^{X}\left( x_{0}\right) \right) $ of the ball $%
B_{r}^{X}\left( x_{0}\right) $ is a convex set and $\ f_{1}\left(
B_{r}^{X}\left( x_{0}\right) \right) \subset int\ B_{r}^{X}\left(
x_{0}\right) $ moreover, the affine space generated over$\ f_{1}\left(
B_{r}^{X}\left( x_{0}\right) \right) $ contains some subspace of $X$ then
the mapping $f$ has a fixed point in the ball $B_{r}^{X}\left( x_{0}\right) $%
.
\end{corollary}

\begin{remark}
\label{R_4} Since the closed convex body is homeomorphic to the closed ball
the cited above corollaries one can transfer to the case of the closed
convex body of the Banach space.
\end{remark}

\section{\label{Sec_4}On solvability of the nonlinear equations and
inclusions}

In the beginning, we will provide the results, which in some sense are
corollaries of results from the above sections, therefore we will lead them
simplified variants.

Let $X$ and $Y$ be $LVTS$, and $f$ be a mapping acting from $X$ to $Y$.

\begin{theorem}
\label{Th_10} Let $y\in Y$ be an element. If there exists such subset $%
G\subseteq D\left( f\right) \subseteq X$ that $f\left( G\right) $ is a
convex subset of $Y$ satisfying the condition i) or ii) of Theorem \ref{Th_7}%
. Then if for $\forall y^{\ast }\in \partial U^{\ast }\subset Y^{\ast }$
there exists such $x\in G$ that fulfills the corresponding inequation 
\begin{equation*}
i)\ \left\langle f\left( x\right) -y,y^{\ast }\right\rangle \cap \left(
0,\infty \right) ;\quad or\quad \left\langle f\left( x\right) -y,y^{\ast
}\right\rangle \cap \left[ 0,\infty \right)
\end{equation*}%
then $y\in f\left( G\right) $, i.e. $\exists x_{1}\in G\Longrightarrow y\in
f\left( x_{1}\right) $ ($f\left( x_{1}\right) =y$).
\end{theorem}

For the proof sufficiently to noted that if denote by $f_{1}$the mapping $%
f_{1}\left( \cdot \right) \equiv f\left( \cdot \right) -y$ on the subset $G$
then it isn't difficult to see that for the mapping $f_{1}:G\longrightarrow
Y $ all conditions of Therem \ref{Th_7} are fulfilled, consequently its
claim fulfills also.

\begin{remark}
\label{R_5} We should be noted the condition "for $\forall y^{\ast }\in
\partial U^{\ast }\subset Y^{\ast }$ there exists such $x\in G$ " be some
relation between $\partial U^{\ast }$ and some subset $G_{0}$ of $G$,
therefore one can denote it as a mapping $g$ acting from $X$ to $Y^{\ast }$
such that for each $y^{\ast }\in \partial U^{\ast }\Longrightarrow g\left(
y^{\ast }\right) =x\in G_{0}$ moreover,$\ g\left( \partial U^{\ast }\right)
=G_{0}$, or it inverse $g^{-1}\left( G_{0}\right) =\partial U^{\ast }$.
Since under investigation of the concrete problem it is necessary to seek
the above-mentioned relation unlike of the brought general results. The
existence of such mapping allows obtaining the necessary a priori estimates.
Therefore in what follows, we will lead results according to this remark.
\end{remark}

\begin{corollary}
\label{C_7} Let the mapping $f$ act from Banach space $X$ to the Banach
space $Y$, where (iii) $Y\equiv X^{\ast }$ or (iv) $Y\equiv X\equiv X^{\ast
\ast }$. Then the equation (the inclusion) $f\left( x\right) =y$ ($f\left(
x\right) \ni y$) is solvable for any $y\in Y$ fulfilling conditions:

1) There exists such subset $G\subseteq D\left( f\right) \subseteq X$ that $%
f\left( G\right) $ is a convex subset with nonempty interior in $Y$;

2) For $\forall x\in \partial U^{\ast }\subset Y^{\ast }$ takes place the
expression (iii) $\left\langle f\left( x\right) -y,x\right\rangle \cap
\left( 0,\infty \right) \neq \varnothing $; (iv) $\left\langle f\left(
x\right) -y,J\left( x\right) \right\rangle \cap \left( 0,\infty \right) \neq
\varnothing $, where $J$ is the duality operator: $J:X\longrightarrow
X^{\ast }$.
\end{corollary}

Now we will prove certain results that generalize the well-known theorems
from articles [15, 32, 40], etc. in the sense that the examined mapping
isn't continuous or completely continuous and the subset on which this
mapping satisfies the needed condition may be arbitrary.

So, let $Y$ be a semi-reflexive $LVTS$ (see, e.g. [5]), $S$ be a weakly
complete "reflexive" $pn-$space (see, [32, 44] and references in these), $X$
be a separable $VTS$, moreover $X\subset S$ and is everywhere dense in $S$
(or $Y$ and $S$ be semi-reflexive $LVTS$ ), $X_{m}$ be an $m-$dimension
linear subspace of $X$, generated over first $m$ elements from totl systems
of $X$.

3) Let $f:S\longrightarrow Y$ be a bounded \footnote{%
Let $X$ and $Y$ be $LVTS$ and $f$ be a mapping that acts from $X$ to $Y$. A
mapping $f$ be called bounded if the image of each bounded subset of $X$ is
a bounded subset of $Y$.} and weakly closed mapping.\footnote{%
Let $X$ and $Y$ be $LVTS$ and $f$ be a mapping that acts from $X$ to $Y$. A
mapping be called weakly closed if a sequence $\left\{ x_{\alpha }\right\} $
from $D\left( f\right) $ weakly converges to $x\in D\left( f\right) $ and
sequence $\left\{ f\left( x_{\alpha }\right) \right\} $ converges to $y\in Y$
then $y=f\left( x\right) $.} Let $G\subset X$ be such neighborhood of zero
the space $X$ that for each $m=1,2,...$ the set $G\cap X_{m}=G_{m}$ is
closed neighborhood of zero the space $X$ and $f\left( G_{m}\right) $ is
closed convex subset in $Y$.

\begin{theorem}
\label{Th_11}Let the condition 3) is fulfilled and $A:X\longrightarrow
Y^{\ast }$ be a linear continuous operator. Then each $y\in Y$ belongs to
the subset $f\left( G\right) -\ker A^{\ast }$ in other words $y+z\in f\left(
G\right) $ i.e. there exist $x_{0}\in G$ such that $f\left( x_{0}\right)
=y+z $ if operator $A:X\longrightarrow Y^{\ast }$ such that the inequation 
\begin{equation}
\left\langle f\left( x\right) -y,Ax\right\rangle \geq 0,\quad x\in \partial
G_{m},\text{ }m=1,2,...\text{\ }  \label{4.2}
\end{equation}%
holds for each $m$, where $z\in \ker A^{\ast }\subset Y$. \ 
\end{theorem}

For the proof is sufficiently noted the proof is led applying the Galerkin
method, which usually is applied in such type cases (see, [15, 40, 44],
etc), however here instead of "acute-angle lemma" is used its generalization
proved in Section \ref{Sec_1}.

\begin{remark}
\label{R_6} The weakly closedness of the mapping in the multivalued case be
understood as: if a sequence $\left\{ x_{\alpha }\right\} $ from $D\left(
f\right) $ weakly converges to $x\in D\left( f\right) $ and correspondent
sequence $\left\{ f\left( x_{\alpha }\right) \right\} $ weakly converges to
the subset $\yen \subset Y$ then $\yen =f\left( x\right) $. Consequently,
the claim of Theorem \ref{Th_11} in this case will be as: each $y\in Y$
belongs to the subset $f\left( G\right) -\ker A^{\ast }$ in other words $%
y+z\in f\left( G\right) $ i.e. there exist $x_{0}\in G$ such that $y+z$ $\in
f\left( x_{0}\right) $ if operator $A:X\longrightarrow Y^{\ast }$ such that
the inequation (\ref{4.2}) holds for each $m$, where $z\in \ker A^{\ast
}\subset Y$.
\end{remark}

\ Now consider equations, and also inclusions in Banach spaces. Let spaces $%
Y $, $S$ be such as above, $X$ be a reflexive separable Banach space
everywhere dense in $S$ and $f:S\longrightarrow Y$ be a bounded mapping.
Consider the following conditions.

c) Let there exists such $r_{0}>0$ that for any closed ball $B_{r}^{X}\left(
0\right) \subset X$ ($r>r_{0}>0$) there exists such neighborhood$\ G_{r}$ of 
$0\in S$ that $B_{r}^{X}\left( 0\right) \subseteq G_{r}$, $G_{r}\cap
B_{r_{1}}^{X}\left( 0\right) \subseteq B_{r_{1}}^{X}\left( 0\right) $ and $%
f\left( G_{r}\right) $ be open (or closed) convex set in $Y$, where $%
r_{1}\left( r\right) \geq r$ ($r_{1}$ dependent only on $r$);

d) There exists such linear bounded operator $A:X\longrightarrow Y^{\ast }$
that the following expression fulfills 
\begin{equation*}
\left\langle f\left( x\right) ,Ax\right\rangle \geq \varphi \left( \left[ x%
\right] _{S}\right) \ \left[ x\right] _{S}\text{ \ }\&\text{ }\varphi \left( %
\left[ x\right] _{S}\right) \nearrow \infty \text{\ \ at \ }\left[ x\right]
_{S}\nearrow \infty
\end{equation*}%
where $\varphi :R_{+}^{1}\longrightarrow $ $R^{1}$ be a continuous function
and $\left[ x\right] _{S}$ be $p-$norm in $S$;

e) There exists such nonlinear operator $g:X\subseteq S\longrightarrow
Y^{\ast }$ that $\frac{g\left( x\right) }{\left\Vert g\left( x\right)
\right\Vert _{Y^{\ast }}}=\widehat{g}\left( x\right) $ saisfies the
condition $\widehat{g}\left( X\right) =S_{1}^{Y^{\ast }}\left( 0\right) $
and on $X$ fulfills the expression 
\begin{equation*}
\left\langle f\left( x\right) ,\widehat{g}\left( x\right) \right\rangle \geq
\varphi \left( \left[ x\right] _{S}\right) \ \left[ x\right] _{S}\text{ \ }\&%
\text{ }\varphi \left( \left[ x\right] _{S}\right) \nearrow \infty \text{\ \
at \ }\left[ x\right] _{S}\nearrow \infty ,
\end{equation*}%
where $\varphi $ and $\left[ x\right] _{S}$ are same as in the condition d).

(In the case when $f:X\longrightarrow Y$ in the above expression instead $p-$%
norm needs to take $\left\Vert x\right\Vert _{X}$.)

\begin{theorem}
\label{Th_12} Let conditions c) and d) fulfill. Then for any $y\in Y$
satisying condition 
\begin{equation}
\sup \left\{ \frac{\left\langle y,Ax\right\rangle }{\left[ x\right] _{S}}%
\left\vert \ x\in X\right. \right\} <\infty  \label{4.3}
\end{equation}%
there exist such $x_{0}$ and $y_{0}\in \ker A^{\ast }\cap Y$ that $f\left(
x_{0}\right) =y+y_{0}$.
\end{theorem}

\begin{proof}
For the proof is sufficient to note that according to condition d) there
exists ball $B_{r_{0}}^{X}\left( 0\right) $ such that on $%
S_{r_{0}}^{X}\left( 0\right) $ fulfills inequation $\left\vert \left\langle
y,Ax\right\rangle \right\vert \leq \varphi \left( \left[ x\right]
_{S}\right) \left[ x\right] _{S}$ due to (\ref{4.3}). Therefore one can use
the condition c). According to condition c) there exist such neighborhood $%
G_{r_{0}}$ of $0$ and ball $B_{r_{01}}^{X}\left( 0\right) $\ that on $%
S_{r_{01}}^{X}\left( 0\right) $ fulfills inequation $\left\langle f\left(
x\right) -y,Ax\right\rangle \geq 0$. Consequently, conditions of Theorem \ref%
{Th_10} satisfy then its claim satisfies also.
\end{proof}

\begin{theorem}
\label{Th_13} Let conditions c) and e) fulfill. Then for any $y\in Y$
satisying condition 
\begin{equation}
\sup \left\{ \frac{\left\langle y,\widehat{g}\left( x\right) \right\rangle }{%
\left[ x\right] _{S}}\left\vert \ x\in X\right. \right\} <\infty  \label{4.4}
\end{equation}%
there exists such $x_{0}$ that $f\left( x_{0}\right) =y$.
\end{theorem}

The proof of this theorem is analogous to the above-provided proof.

\begin{remark}
\label{R_7} We should be noted one can prove theorems of the previous
theorems type in the case when $Y$, $X$ be Banach spaces moreover, $Y$ is
reflexive space and $f:X\longrightarrow Y$ be a bounded mapping, as these
are led by the analogous way, therefore we do not to adduce their here.
\end{remark}

We provide now one result for the equation with the odd operator.

\begin{theorem}
\label{Th_14} Let $f$ acts from reflexive Banach space $X$ to its dual space 
$X^{\ast }$ and is the single-value odd operator. Assume there exists such
closed balanced convex neighborhood $G\subset X$ that $f\left( G\right) $ is
a convex closed subset of $X^{\ast }$. Then there exist such subset $%
G_{1}\subseteq G$ and a subspace $X_{0}^{\ast }\subseteq $\ $X^{\ast }$ that
for each $x^{\ast }\in X_{0}^{\ast }$ satisfying of inequation $\left\Vert
x^{\ast }\right\Vert _{X^{\ast }}\leq \left\Vert f\left( x\right)
\right\Vert _{X^{\ast }}$, $\forall x\in G_{1}$ the equation $f\left(
x\right) =x^{\ast }$ is solvable in $G$.
\end{theorem}

\begin{proof}
Due to the condition of theorem $0\in f\left( G\right) $. Then according to
the convexity of the $f\left( G\right) $ the affine space $X_{0}^{\ast }$
generated over $f\left( G\right) $ is a subspace of\ $X^{\ast }$, and $%
f\left( G\right) $ is the convex closed body in the subspace $X_{0}^{\ast }$%
. Whence follows that without loss of generality one can suppose that $%
X_{0}^{\ast }\equiv X^{\ast }$. Since $X$ and $X^{\ast }$ are the reflexive
Banach spaces, one can assume that these are the strongly convex spaces, and
the dual mapping $J$: $X\underset{J^{-1}}{\overset{J}{\longleftarrow
\longrightarrow }}X^{\ast }$ is the monotone biection (one-to-one, see, e.g.
[31, 40]).

As $f\left( G\right) $ is the closed convex set there exists a subset $G_{1}$
of $G$ for which takes place the relations $f\left( G_{1}\right) =\partial
f\left( G\right) $, $f^{-1}\left( \partial f\left( G\right) \right)
\supseteq G_{1}$. Another hand using the dual mapping $J$ one can determine
a subset $G_{0}=J^{-1}\left( \partial f\left( G\right) \right) $. Then there
exists such one-to-one mapping $f_{0}$ acting in $X$ that $%
f_{0}:G_{1}\longleftarrow \longrightarrow G_{0}$ and for each $x\in G_{1}$
fulfill the equation 
\begin{equation*}
\left\langle f\left( x\right) ,f_{0}\left( x\right) \right\rangle
=\left\langle f\left( x\right) ,J^{-1}\left( f\left( x\right) \right)
\right\rangle =\left\Vert f\left( x\right) \right\Vert _{X^{\ast
}}\left\Vert J^{-1}\left( f\left( x\right) \right) \right\Vert _{X}\ .
\end{equation*}

According to the condition of the theorem, the set $\partial f\left(
G\right) $ is the boundary of the closed convex absorbing subset of $X^{\ast
}$ therefore, the subset $G_{0}$ also will be a boundary of the closed
absorbing neighborhood of zero of $X$. Then using of Theorem \ref{Th_10}, we
obtain the solvability of the equation $f\left( x\right) =y$ for any $y\in
X^{\ast }$ satisfying the inequality $\left\langle f\left( x\right)
-y,J^{-1}\left( f\left( x\right) \right) \right\rangle \geq 0$. Due to the
above reasoning, the equation $f\left( x\right) =y$ is solvable for any $%
y\in X^{\ast }$ satisfying the inequality 
\begin{equation*}
\left\langle y,J^{-1}\left( f\left( x\right) \right) \right\rangle \leq
\left\langle f\left( x\right) ,J^{-1}\left( f\left( x\right) \right)
\right\rangle =\left\Vert f\left( x\right) \right\Vert _{X^{\ast
}}\left\Vert J^{-1}\left( f\left( x\right) \right) \right\Vert _{X}\ .
\end{equation*}
\end{proof}

\section{\label{Sec_5}Some sufficient conditions on the convexity of the
image of mappings}

Let $X$, $Y$ be the $VTS$, and $f$ be a mapping acting from $X$ to $Y$.

\begin{lemma}
\label{L_7} Let the mapping $f$, acting from $LVTS$ $X$ to $LVTS$ $Y$, some
subsets from $D\left( f\right) \subseteq X$ translated to connected subsets
of $Y$. Let there exist such subsets $G_{1}\subset G\subseteq D\left(
f\right) $ that $f\left( G_{1}\right) $, $f\left( G\right) $ be connected
subsets with the nonempty interiors of $Y$ and such convex subset $M$ of $Y$
that\ inequations $f\left( G_{1}\right) \subseteq M\subseteq f\left(
G\right) $ hold. Then there exists such subset $G_{2}\subset G$ that $%
f\left( G_{2}\right) =M$ (if the mapping $f$ is multi-valued then one can
determine the mapping $f_{1}$ (or restrict this mapping) as follows $%
f_{1}\left( x\right) =f\left( x\right) \cap M$ for each $x\in G_{2})$.
\end{lemma}

The proof is obvious.

\begin{proposition}
\label{Pr_2} Let the mapping $f$, acting from $LVTS$ $X$ to $LVTS$ $Y$, be
locally homeomorphic from $D\left( f\right) \subseteq X$ on $\func{Im}\left(
f\right) \subseteq Y$. Then if the $D\left( f\right) $ contains a subset
with a nonempty interior then the mapping $f$ is fulfilled the condition of
Lemma \ref{L_7}.
\end{proposition}

\begin{lemma}
\label{L_8} Let $X$ be reflexive Banach space, and $f$ acting from $X$ to
its dual space $X^{\ast }$, be a potentially mapping with a convex potential 
$F$ (the Gateaux derivative is $\partial F=f$). Then the image $f\left(
G\right) $ of each convex subset $G$, generated by potential be a convex
subset of $X^{\ast }$. (Where a subset $G$, generated by potential
understood as $G=\left\{ x\in X\left\vert \ F\left( x\right) \leq C,\right.
\right\} $, where $C$ is the constant.)
\end{lemma}

\begin{proof}
According to the condition on $f$ is the Gateaux derivative of a
differentiable convex functional $F$, i.e. $\partial F\left( x\right)
=f\left( x\right) $ for each $x\in int\ D\left( f\right) $. Well-known that
the dual functional $F^{\ast }$ to $F$ also is convex functional (see, [12,
16, 19, 25, 26, 27]). Moreover, the following relations 
\begin{equation}
\left( i\right) \ \forall x\in domF\Longrightarrow f\left( x\right) \in
domF^{\ast };\quad \left( ii\right) \ \forall x^{\ast }\in \partial F\left(
x\right) \Longleftrightarrow x\in \partial F^{\ast }\left( x^{\ast }\right)
\label{5.1}
\end{equation}%
are fulfilled.

Where $\partial F$ is is the subdifferential of the functional $F$ that in
this case is the Gateaux derivative of $F$ (see, [16, 19]). Then if $%
G\subset int\ domF$ is a convex closed subset then the corresponding subset $%
G^{\ast }\subset int\ domF^{\ast }$ also is a convex closed subset. In the
other words, the image $f\left( G\right) $ of the subset $G$ is the $G^{\ast
}$, i.e. $f\left( G\right) =G^{\ast }$.
\end{proof}

Whence implies the correctness of the following result.

\begin{corollary}
\label{C_8} Let $X$ be reflexive Banach space, and $f$ acts from $X$ to its
dual space $X^{\ast }$, then if the mapping $f$ is the subdifferential of a
convex functional then the claim of the Lemma \ref{L_8} is true.
\end{corollary}

Now we will lead one result for the case when $X$ and $Y$ are the spaces of
functions, and the mappings are concrete chosen.

\begin{lemma}
\label{L_9} Let the bounded mapping $f$ acting in $R^{1}$ the connected
subset translates to connected subset and satisfy the following inequalities 
\begin{equation*}
f\left( \xi \right) \cdot \xi \geq c_{0}\left\vert \xi \right\vert ^{p+1}-%
\widehat{c}_{0};\quad \left\vert f\left( \xi \right) \right\vert \leq
c_{1}\left\vert \xi \right\vert ^{p}+\widehat{c}_{1},\quad \forall \xi \in
R^{1},
\end{equation*}%
for some constants $c_{0}$, $c_{1}>0$, $\widehat{c}_{0},\widehat{c}_{1}\geq
0 $, $p>1$.

Assume $A:W_{0}^{k,p_{0}}\left( \Omega \right) \longrightarrow
L_{p_{0}}\left( \Omega \right) $ is the linear continuous operator
satisfying the following inequation 
\begin{equation*}
c_{2}\left\Vert u\right\Vert _{W_{0}^{k,p_{0}}}-\widehat{c}_{2}\leq
\left\Vert A\left( u\right) \right\Vert _{L_{p_{0}}}\leq c_{3}\left\Vert
u\right\Vert _{W_{0}^{k,p_{0}}}+\widehat{c}_{3},\quad \forall u\in
W_{0}^{k,p_{0}}\left( \Omega \right) ,
\end{equation*}%
where $\Omega \subset R^{n}$, $n\geq 1$, is the bounded domain with
suffitiently smooth boundary, and $c_{2}$,$c_{3}>0$, $\widehat{c}_{2},%
\widehat{c}_{3}\geq 0$, $p_{0}\geq p\cdot p_{1}$, $p,p_{1}>1$, $k\geq 0$ are
constants.

Then for each ball $B_{r}^{W_{0}^{k,p_{0}}}\left( 0\right) \subset
W_{0}^{k,p_{0}}\left( \Omega \right) $, $r\geq r_{0}>0$ there exists such
subset $G_{r}\subset W_{0}^{k,p_{0}}\left( \Omega \right) $ that $%
f_{0}\left( G_{r}\right) \equiv \left( f\circ A\right) \left( G_{r}\right) $ 
$\equiv f\left( A\left( G_{r}\right) \right) $ is a convex subset with the
nonempty interior of $L_{p_{1}}\left( \Omega \right) $ moreover, $%
f_{0}\left( B_{r}^{W_{0}^{k,p_{0}}}\left( 0\right) \right) \subseteq
f_{0}\left( G_{r}\right) \subseteq f_{0}\left(
B_{r_{1}}^{W_{0}^{k,p_{0}}}\left( 0\right) \right) $ holds for some $%
r_{1}\geq r$, where $r_{0}$ is a constant depending at the above constants.
\end{lemma}

\begin{proof}
It isn't difficult to see that the linear operator $A$ is surjective (see,
e.g. [5, 15, 19, 26, 29]) therefore, it is enough to lead proof for the
mapping $f:$ $L_{p_{0}}\left( \Omega \right) \longrightarrow L_{p_{1}}\left(
\Omega \right) $. Consider the following subset of $L_{p_{0}}\left( \Omega
\right) $ 
\begin{equation*}
M_{r_{2}}\equiv \left\{ u\in L_{p_{0}}\left( \Omega \right) \left\vert \
\left\Vert f\left( u\right) \right\Vert _{L_{p_{1}}}\leq r_{2},\
r_{2}>r_{0}\right. \right\} .
\end{equation*}%
Whence using conditions and providing some estimations we obtain 
\begin{equation*}
\overline{c}_{0}\left\Vert u\right\Vert _{L_{p_{0}}}^{p}-\widetilde{c}%
_{0}\leq \left\Vert f\left( u\right) \right\Vert _{L_{p_{1}}}\leq \overline{c%
}_{1}\left\Vert u\right\Vert _{L_{p_{0}}}^{p}+\widetilde{c}_{1}.
\end{equation*}

These inequations show that the image $f\left( M_{r_{2}}\right) $ of the
above-introduced subset $M_{r_{2}}$ contains some ball and belongs to
another ball from $L_{p_{1}}\left( \Omega \right) $. Consequently, there
exist such subsets $G_{r}$ that $f_{0}\left( G_{r}\right) $ will be convex
subsets with the nonempty interior of $L_{p_{1}}\left( \Omega \right) $.
\end{proof}

We will provide some well-known facts from [5] that are necessary for the
next result.

\begin{definition}
\label{D_1}(see, [5]) Let $K$ be a convex set of the linear space ($LS$) $X$
containing zero as the $C-$interior point. If $\mu $ is the Minkowski
functional of $K$, then the function defined for all pair $x,\ y\in X$ by
the equation 
\begin{equation*}
\tau \left( x,y\right) =\underset{\alpha \searrow +0}{\lim }\ \frac{1}{%
\alpha }\left[ \mu \left( x+\alpha y\right) -\mu \left( x\right) \right]
\end{equation*}%
called tangential functional of the set $K$.
\end{definition}

We should be noted if $K$ is the convex set as above then $\frac{1}{\alpha }%
\left[ \mu \left( x+\alpha y\right) -\mu \left( x\right) \right] $ is the
growing function at real $\alpha >0$ and the above-mentioned limit exists
for all pair $x,\ y\in X$ (see, [5]). Let $K$ is a subset of the $LS$ $X$
and the point $x\in K$ is the $C-$boundary point of $K$ then the functional $%
x^{\ast }\in X^{\ast }$ called the tangent to subset $K$ on the point $x\in
K $ if there exists such constant $c$ that $\left\langle x^{\ast
},x\right\rangle =c$ and $\left\langle x^{\ast },y\right\rangle \leq c$ for $%
\forall y\in K$.

\begin{theorem}
\label{Th_15}(see, [5]). Let $X$ be a $VTS$ and $K\subset X$ be a closed
subset of $X$ possessing interior points. Assume $K$ possesses a nontrivial
tangent functional on all points of an everywhere dense subset of the
boundary of $K$ then $K$ is a convex set.
\end{theorem}

For the proof, and also about the correctness of its inverse statement see,
[5].

\begin{corollary}
\label{C_9} Let a bounded mapping $f$ acting from reflexive Banach space $X$
to its dual space $X^{\ast }$ is (i) a monotone hemi-continuous coercive
operator (see, [15, 42]). Then $f$ translates a closed convex subset with a
nonempty interior of $X$ onto a closed convex subset with a nonempty
interior of $X^{\ast }$; (ii) a positive homogeneous radially continuous
monotone mappings. Then $f$ translates a convex subset of $X$ defined by a
functional depending on mapping $f$, for which zero is an interior point,
onto a convex subset of $X^{\ast }$.
\end{corollary}

The proof of the (i) follows from Theorem \ref{Th_15}, and the proof of the
(ii) immediately follows from the presentation of the corresponding
functional 
\begin{equation*}
F\left( x\right) \equiv \underset{0}{\overset{1}{\int }}\left\langle \
f\left( tx\right) ,x\right\rangle dt\equiv \frac{1}{\alpha +1}\left\langle \
f\left( x\right) ,x\right\rangle ,
\end{equation*}%
where $\alpha $ is a exponent of homogeneity.

\section{\label{Sec_6}Examples}

1. Let $X$ be a reflexive Banach space, $J$ be a duality operator $%
J:X\longrightarrow X^{\ast }$, generated by a strongly monotone growing
continuous function (ee, e.g. [15]) 
\begin{equation*}
\Phi :R_{+}^{1}\longrightarrow R_{+}^{1},\quad \Phi \left( 0\right) =0,\quad
\Phi \left( \tau \right) \nearrow \infty \text{ at }\tau \nearrow \infty ,
\end{equation*}%
$\varphi :R_{+}^{1}\longrightarrow R_{+}^{1}$ be some mapping translating
the connected set to connected set (i.e. connected mapping) that satisfies
the condition: for each interval $I\subset R_{+}^{1}$ there exists $\sup
\left\{ \varphi \left( \tau \right) \left\vert \ \tau \in I\right. \right\}
=\varphi \left( \tau _{I}\right) $, moreover, here maybe $\varphi \left(
\tau _{I}\right) =\infty $.

Let $X$ and $X^{\ast }$\ be strongly convex spaces, $f:$\ $X\longrightarrow
X^{\ast }$ be a mapping having the form as in the following equation 
\begin{equation}
f\left( x\right) \equiv \varphi \left( \left\Vert x\right\Vert _{X}\right) \
J\left( x\right) =y,\quad y\in X^{\ast }.  \label{6.1}
\end{equation}

\begin{theorem}
\label{Th_16} Under the above conditions the equation (\ref{6.1}) solvable
for any $y\in X^{\ast }$ satisfies of the following inequation 
\begin{equation*}
\left\Vert y\right\Vert _{X^{\ast }}\leq \sup \left\{ \varphi \left(
r\right) \Phi \left( r\right) \left\vert \ r\geq 0\right. \right\} =\varphi
\left( r_{0}\right) \Phi \left( r_{0}\right) .
\end{equation*}
\end{theorem}

\begin{proof}
In the beginning note that according to Lemma \ref{L_8} the image $f\left(
B_{r}^{X}\left( 0\right) \right) \subset X^{\ast }$ of the ball $%
B_{r}^{X}\left( 0\right) \subset X$ is a convex subset of $X^{\ast }$.
According to spaces $X$ and $X^{\ast }$ be strongly convex, therefore the
duality opertator $J$ is a bijective mapping (see, [5, 15, 23]), moreover
the image of each ball $B_{r_{0}}^{X}\left( 0\right) \subset X$ is a ball $%
B_{r_{1}}^{X^{\ast }}\left( 0\right) \subset X^{\ast }$, where $r_{1}=\Phi
\left( r_{0}\right) $. Whence implies that the mapping $f$ can be
represented as 
\begin{equation}
f\left( x\right) \equiv \varphi \left( \left\Vert x\right\Vert _{X}\right) \
\Phi \left( \left\Vert x\right\Vert _{X}\right) \ x^{\ast },\quad \forall
x\in B_{r}^{X}\left( 0\right) ,  \label{6.2}
\end{equation}%
as $\ J\left( x\right) \equiv \Phi \left( \left\Vert x\right\Vert
_{X}\right) \ x^{\ast }$, where the functional $x^{\ast }\in S_{1}^{X^{\ast
}}\left( 0\right) \subset X^{\ast }$ is the functional generating the norm $%
\left\Vert x\right\Vert _{X}$ .

Thus the image $f\left( S_{r_{0}}^{X}\left( 0\right) \right) $\ of each
sphere $S_{r_{0}}^{X}\left( 0\right) \subset X$ be a the sphere $%
S_{r_{1}}^{X^{\ast }}\left( 0\right) \subset X^{\ast }$ under mapping $f$.
Consequently, since $x^{\ast }\in S_{1}^{X^{\ast }}\left( 0\right) \subset
X^{\ast }$ from (\ref{6.2}) implies that the image $f\left(
B_{r_{0}}^{X}\left( 0\right) \right) $\ of each ball $B_{r_{0}}^{X}\left(
0\right) \subset X$ be a the ball $B_{r_{1}}^{X^{\ast }}\left( 0\right)
\subset X^{\ast }$ with the corresponding radius $r_{1}=r_{1}\left(
r_{0}\right) $.

Then use the Theorem \ref{Th_10} we get that the equation (\ref{6.1}) is
solvable for any $y\in X^{\ast }$ satisfying inequation 
\begin{equation*}
\left\vert \left\langle y,x\right\rangle \right\vert \leq r\varphi \left(
r\right) \Phi \left( r\right) ,\quad \forall x\in S_{r}^{X}\left( 0\right)
,\ r>0.
\end{equation*}

In other words, the equation (\ref{6.1}) is solvable for any $y\in X^{\ast }$
satisfying follows there exists such $r>0$ that $\left\Vert y\right\Vert
_{X^{\ast }}\leq \varphi \left( r\right) \Phi \left( r\right) $.
\end{proof}

2. On the bounded domain $\Omega \subset R^{n}$, $n\geq 1$, with
sufficiently smooth boundary $\partial \Omega $ consider the following
problem 
\begin{equation}
f\left( u\right) \equiv -\Delta u+\psi \left( u\right) \ni h\left( x\right)
,\quad x\in \Omega ,\quad u\left\vert \ _{\partial \Omega }=0,\right.
\label{6.3}
\end{equation}%
where $\Delta $ is the laplacian, $\psi $ be some mapping acting from $%
H_{0}^{1}\equiv W_{0}^{1,2}\left( \Omega \right) $ (see, [41]) to $%
L_{2}\left( \Omega \right) $ such that for each function $u\in H_{0}^{1}$
the image $\psi \left( u\right) $ is the set of functions $\left\{ v\left(
x\right) \right\} \subset L^{\infty }\left( \Omega \right) $ that has the
following representation 
\begin{equation*}
v\left( x\right) =\left\{ 
\begin{array}{c}
1,\quad x\in \left\{ y\in \Omega \left\vert \ u\left( x\right) >0\right.
\right\} \\ 
w\in L^{\infty }\left( \Omega \right) ,\quad x\in \Omega _{0}\equiv \left\{
y\in \Omega \left\vert \ u\left( x\right) =0\right. \right\} \\ 
-1,\quad x\in \left\{ y\in \Omega \left\vert \ u\left( x\right) <0\right.
\right\}%
\end{array}%
,\right.
\end{equation*}%
moreover $\left\vert w\left( x\right) \right\vert \leq 1$ a.e. on $\Omega
_{0}$, and $\left[ -1,1\right] \subset \func{Im}\psi $. Here $%
W_{0}^{1,2}\left( \Omega \right) $\ is the Sobolev space of functions and $%
W^{-1,2}\left( \Omega \right) $ is its dual space.

In other words, mapping $f$ acts from $H_{0}^{1}$ to $H^{-1}\equiv
W^{-1,2}\left( \Omega \right) $, in this case, inclusion (\ref{6.3}) is
understood in the following sense the image $f\left( u\right) $ of each
function $u\in H_{0}^{1}$ is a set $-\Delta u+\left\{ v\left( x\right)
\right\} $.

\begin{theorem}
\label{Th_17} Let the above conditions on the problem (\ref{6.3}) be
fulfilled. Then for any $h\in H^{-1}$ there exists a solution $u\left(
x\right) $ of the problem (\ref{6.3}), that belongs to the space $H_{0}^{1}$.
\end{theorem}

\begin{proof}
For the proof enough to show that all condition of the Theorem \ref{Th_10}
fulfills for the mapping $f$. We will show that for any $h\in H^{-1}$ there
exists such set $G\subset H_{0}^{1}$ that $f\left( G\right) $ be a convex
body and exists such subset of $G$ that be a boundary $\partial G_{1}$ of an
absorbing subset $G_{1}$ on which takes place the inequation 
\begin{equation*}
\left\langle f\left( u\right) -h,u\right\rangle >0,\quad \forall u\in
\partial G_{1}\subset G\subset H_{0}^{1}.
\end{equation*}%
Here we will use the Lemma \ref{L_7}, namely here enough to choose such
balls $B_{r_{1}}^{H_{0}^{1}}\left( 0\right) $, $B_{r_{2}}^{H_{0}^{1}}\left(
0\right) \subset $ $H_{0}^{1}$ ($0<r_{1}<r_{2}$) that satisfy the inequation 
$f\left( B_{r_{1}}^{H_{0}^{1}}\left( 0\right) \right) \subseteq M\subseteq $ 
$f\left( B_{r_{2}}^{H_{0}^{1}}\left( 0\right) \right) $, where $M$ be a
convex subset. We should be noted choosing of balls is dependent on the
given $h\in H^{-1}$.

In the beginning, is necessary to show the correctness of some inequalities.
(Relative to the connectivity of the image of the connective set under
mappings that don't are continuous can be to lead the following result 
\footnote{%
Let $X$, $Y$ be linearly connected $LVTS$, and $f:D\left( f\right) \subseteq
X\longrightarrow Y$ be a mapping, $G\subseteq D\left( f\right) $ be a
locally connected subset. Introduce a class of subset of $Y$ and also the
corresponding subset of $R^{1}$ 
\begin{equation*}
\aleph _{f}\left( \gamma _{X},y^{\ast }\right) =\left\{ \left\langle
y,y^{\ast }\right\rangle =\tau \in R^{1}\left\vert \ y\in f\left( x\right)
,\ x\in \gamma _{X}\left( x_{1},x_{2}\right) ,\ x_{1},x_{2}\in G\right.
\right\} ,
\end{equation*}%
where $x_{1},x_{2}\in G$ be some points, and $\gamma _{X}\left(
x_{1},x_{2}\right) \subset G$ be a curve connecting of these points, $%
y^{\ast }\in Y^{\ast }$ be a linear continuous functional.
\par
\begin{theorem}
\label{Th_s}(see, [32]and its references) Let $X$, $Y$ be linearly connected 
$LVTS$, and $f:D\left( f\right) \subseteq X\longrightarrow Y$ be a mapping, $%
G\subseteq D\left( f\right) $ be a locally connected subset. Then if for any 
$x_{1},x_{2}\in G$ there exists such curve $\gamma _{X}\left(
x_{1},x_{2}\right) $\ that the subset $\aleph _{f}\left( \gamma _{X},y^{\ast
}\right) $ be connected\ for each linear continuous functional $y^{\ast }\in
Y^{\ast }$ then $f\left( G\right) $ be a connected set.
\end{theorem}
\par
See, Soltanov K. N., On connectivity of sets and the image of a
discontinuous mappings. On nonlinear mappings. Proc. Inter. Topol. Conf.,
Baku-87, 1989, v II, 166-173 (in russ.)}.)

It isn't difficult to see 
\begin{equation*}
\left\langle f\left( u\right) ,u\right\rangle =\left\langle -\Delta u+\psi
\left( u\right) ,u\right\rangle =\left\Vert \nabla u\right\Vert
_{H}^{2}+\left\Vert u\right\Vert _{L_{1}}\geq \left\Vert \Delta u\right\Vert
_{H^{-1}}^{2},
\end{equation*}%
where $\nabla \equiv \left( \frac{\partial }{\partial x_{1}},...,\frac{%
\partial }{\partial x_{n}}\right) $, \ $\Delta \equiv \nabla ^{\ast }\circ
\nabla $, $H\equiv L_{2}\left( \Omega \right) $.\ Moreover, there exists
constant $c\geq 1$ such that 
\begin{equation}
\left\langle f\left( u\right) ,u\right\rangle =\left\Vert \nabla
u\right\Vert _{H}^{2}+\left\Vert u\right\Vert _{L_{1}}\leq c\left(
\left\Vert \Delta u\right\Vert _{H^{-1}}^{2}+\left\Vert \psi \left( u\right)
\right\Vert _{H^{-1}}\right)  \label{6.4}
\end{equation}%
and on each sphere $S_{r_{1}}^{H_{0}^{1}}\left( 0\right) \subset $ $%
H_{0}^{1} $, $r>1$ for some constants $c_{1},c_{2}>0$\ takes place 
\begin{equation}
\left\Vert \Delta u\right\Vert _{H^{-1}}^{2}\leq c_{1}\left\Vert f\left(
u\right) \right\Vert \leq \left\Vert \Delta u\right\Vert
_{H^{-1}}^{2}+\left\Vert \psi \left( u\right) \right\Vert _{H^{-1}}\leq
c_{2}\left\Vert \Delta u\right\Vert _{H^{-1}}^{2}.  \label{6.5}
\end{equation}

According to the proof of Lemma \ref{L_8}, for the proof that the image $%
f\left( B_{r}^{H_{0}^{1}}\left( 0\right) \right) \subset H^{-1}$ of a ball $%
B_{r}^{H_{0}^{1}}\left( 0\right) \subset H_{0}^{1}$, $r\geq 1$ be a convex
set with the nonempty interior we will use the convexity of the
corresponding functional (see, [16, 19]). Since mapping $f$ is the
subdifferential of the convex functional 
\begin{equation*}
\Phi \left( u\right) \equiv \frac{1}{2}\left\Vert \nabla u\right\Vert
_{H^{n}}^{2}+\left\Vert u\right\Vert _{L_{1}}\equiv \Phi _{0}\left( u\right)
+\Phi _{1}\left( u\right)
\end{equation*}%
due to well-known results (see, [16, 19]) one can assume ball $%
B_{r}^{H_{0}^{1}}\left( 0\right) \subset H_{0}^{1}$ be an effective set of
the functional $\Phi $. Consequently, it is sufficient to examine an
effective set of dual-functional $\Phi ^{\ast }$ of functional $\Phi $. Due
to sub-differentiability of the functional $\Phi $, at least on the $int\
dom\ \Phi $, the inclusions $int\ dom\ \Phi \subseteq \func{Im}\left(
\partial \Phi \right) \subset dom\ \Phi ^{\ast }$ hold.

As the functional is the sum of functionals $\Phi _{0}$ and $\Phi _{1}$ its
dual $\Phi ^{\ast }$ be an infimal convolution of functionals $\Phi _{0}$
and $\Phi _{1}$ therefore, it is necessary to define their dual functionals.
It is known that (see, [30]) under $v^{\ast }\in dom\Phi _{0}^{\ast }$ we
have $\Phi _{0}^{\ast }\left( v^{\ast }\right) \equiv \frac{1}{2}\left\Vert
v^{\ast }\right\Vert _{H^{n}}^{2}$, moreover 
\begin{equation*}
dom\Phi _{0}^{\ast }\equiv \left\{ v^{\ast }\in H^{n}\left\vert \ \left\Vert
v^{\ast }\right\Vert _{H^{n}}\leq r\right. \right\} ,
\end{equation*}%
and under $u^{\ast }\in dom\Phi _{1}^{\ast }\subset L^{\infty }\left( \Omega
\right) $ we have $\Phi _{1}^{\ast }\left( u^{\ast }\right) \equiv {\large %
\varkappa }\left( u^{\ast }\left\vert \ B_{1}^{L^{\infty }}\left( 0\right)
\right. \right) $, where\footnote{${\large \varkappa }\left( u^{\ast
}\left\vert \ M\right. \right) =\QATOPD\{ . {0\ \text{ if \ }u^{\ast }\in
M,}{\infty \text{ \ if \ }u^{\ast }\notin M}$} $B_{1}^{L^{\infty }}\left(
0\right) $ is the closed ball with the radius $r=1_{,}$ with the center in
zero of $L^{\infty }\left( \Omega \right) $, ${\large \varkappa }\left(
u^{\ast }\left\vert \ M\right. \right) $ be an indicator function of the set 
$M$ $\equiv B_{1}^{L^{\infty }}\left( 0\right) \equiv dom\Phi _{1}^{\ast }$.

Then we get for $h\in dom\Phi ^{\ast }\subset H^{-1}$ 
\begin{equation*}
\Phi ^{\ast }\left( h\right) \equiv \Phi _{0}^{\ast }\left( h^{\ast }\right)
+\Phi _{1}^{\ast }\left( u^{\ast }\right) h^{\ast }+u^{\ast }=h
\end{equation*}%
holds, and also $dom\Phi ^{\ast }=dom\Phi _{0}^{\ast }+dom\Phi _{1}^{\ast }$%
, moreover is known $\left( \Phi _{0}\circ \nabla \right) ^{\ast }=\nabla
^{\ast }\circ \Phi _{0}^{\ast }$ ([19]), in this case $\left( \Phi _{0}\circ
\nabla \right) ^{\ast }\left( h^{\ast }\right) =\Phi _{0}^{\ast }\left(
v^{\ast }\right) $, $\nabla ^{\ast }v^{\ast }=h^{\ast }$ ([16]).

Thus we obtain 
\begin{equation*}
\Phi ^{\ast }\left( h\right) \equiv \underset{u^{\ast },h^{\ast }}{\inf }%
\left\{ 
\begin{array}{c}
\frac{1}{2}\left\Vert v^{\ast }\right\Vert _{H^{n}}^{2}\left\vert \ h^{\ast
}+u^{\ast }=h,\quad h^{\ast }=\nabla ^{\ast }v^{\ast },\right. \\ 
h^{\ast }\in B_{1}^{H^{-1}}\left( 0\right) \subset H^{-1},u^{\ast }\in
B_{1}^{L^{\infty }}\left( 0\right) \subset L^{\infty }\left( \Omega \right)%
\end{array}%
\right\} .
\end{equation*}

Whence imply that $int\ dom\Phi ^{\ast }\neq \varnothing $, moreover $%
B_{r}^{H^{-1}}\left( 0\right) \subseteq dom\Phi ^{\ast }$. Consequently, $%
f\left( B_{r}^{H_{0}^{1}}\left( 0\right) \right) $ have nonempty interior,
and also other condition of Lemma \ref{L_8} fulfills for the mapping $f$ by
virtue of inequalities (\ref{6.4}) and (\ref{6.5}). Moreover, the $dom\Phi
^{\ast }$ expands according to the growth of the radius $r$, which proves
the correctness of the claim of Theorem\ \ref{Th_17}.
\end{proof}

3. Now we will lead one simple example on the application of one of the
fixed-point theorems that proved in this work.

Let $\Omega \subset R^{n}$ ($n\geq 1$) be a bounded domain with sufficiently
smooth boundary $\partial \Omega $, $H\equiv L_{2}\left( \Omega \right) $ be
a Lebesgue space that is the Gilbert space and $f:D\left( f\right) \subseteq
H\longrightarrow H$ be a nonlinear mapping. Assume the mapping $f$ as acting
in $H$ has the representation $f\left( u\right) \equiv u-\alpha \left\Vert
u-u_{0}\right\Vert _{H}^{2}\ \left( u-u_{0}\right) $, where $0<\alpha \leq
4^{-1}$. We want that $f:B_{1}^{H}\left( u_{0}\right) \subset
H\longrightarrow B_{1}^{H}\left( u_{0}\right) $ holds. Denote by $\widetilde{%
u}=u-u_{0}$ for any $u\in B_{1}^{H}\left( u_{0}\right) $ then $f\left(
u\right) \equiv u_{0}+\widetilde{u}-\alpha \left\Vert \widetilde{u}%
\right\Vert _{H}^{2}\widetilde{u}$. We will show that $\left\Vert f\left(
u\right) -u_{0}\right\Vert _{H}\leq 1$ 
\begin{equation*}
\left\Vert f\left( u\right) -u_{0}\right\Vert _{H}=\left\Vert u_{0}+%
\widetilde{u}-\alpha \left\Vert \widetilde{u}\right\Vert _{H}^{2}\widetilde{u%
}-u_{0}\right\Vert _{H}=\left\Vert \left( 1-\alpha \left\Vert \widetilde{u}%
\right\Vert _{H}^{2}\right) \widetilde{u}\right\Vert _{H}.
\end{equation*}%
Whence according to the definition of $\widetilde{u}$ fulfills $\left\Vert 
\widetilde{u}\right\Vert _{H}=\left\Vert u-u_{0}\right\Vert _{H}\leq 1$\
then we have 
\begin{equation*}
\left\Vert f\left( u\right) -u_{0}\right\Vert _{H}=\left\Vert \left(
1-\alpha \left\Vert \widetilde{u}\right\Vert _{H}^{2}\right) \widetilde{u}%
\right\Vert _{H}=\left( 1-\alpha \left\Vert \widetilde{u}\right\Vert
_{H}^{2}\right) \left\Vert \widetilde{u}\right\Vert _{H}<1
\end{equation*}%
due to condition on the $\alpha $. Hence one can claim for any $u\in
B_{1}^{H}\left( 0\right) $ the $\left\Vert f\left( u\right)
-u_{0}\right\Vert _{H}<1$, in the other words, $f\left( B_{1}^{H}\left(
u_{0}\right) \right) \subset B_{1}^{H}\left( u_{0}\right) $ holds.

Thus, if define the mapping $f_{1}\left( u\right) =u-f\left( u\right) $ then
we get 
\begin{equation*}
f_{1}\left( u\right) =u-f\left( u\right) =\alpha \left\Vert
u-u_{0}\right\Vert _{H}^{2}\left( u-u_{0}\right) =\alpha \left\Vert 
\widetilde{u}\right\Vert _{H}^{2}\widetilde{u}
\end{equation*}%
consequently, $f_{1}$\ be the Gateaux derivative of a convex functional $%
F_{1}$, i.e. $f_{1}\left( u\right) =\partial F_{1}\left( u\right) $, where $%
F_{1}\left( u\right) =\frac{1}{4}\alpha \left\Vert u-u_{0}\right\Vert
_{H}^{4}=\frac{1}{4}\alpha \left\Vert \widetilde{u}\right\Vert _{H}^{4}$.
According to the results of the above section, the image $f$ $\left(
B_{1}^{H}\left( u_{0}\right) \right) $ of the mapping $f$ is a convex set
with a nonempty interior. On the other hande, the necessary inequality
fulfills as for each $\widetilde{u}\in S_{1}^{H}\left( 0\right) $ we have 
\begin{equation*}
\left\langle f_{1}\left( u\right) ,\widetilde{u}\right\rangle =\left\langle
\alpha \left\Vert \widetilde{u}\right\Vert _{H}^{2}\widetilde{u},\widetilde{u%
}\right\rangle =\alpha \left\Vert \widetilde{u}\right\Vert _{H}^{4}>0.
\end{equation*}

Consequently, the mapping $f$ possesses in $B_{1}^{H}\left( u_{0}\right) $ a
fixed-point, i.e. there exists $u_{1}\in B_{1}^{H}\left( u_{0}\right) $ such
that $f\left( u_{1}\right) =u_{1}$.


\begin{thebibliography}{99}
\bibitem{1} Poincare H., Sur un theoreme de geometrie, Rendiconti Circolo
mat., Plermo, 1912, 33, 375-407

\bibitem{2} Dugundji J., Granas A., Fixed point theory, Springer Monographs
in Mathematics, Springer New York, 2003

\bibitem{3} Nirenberg L., Topics in Nonlinear Functional Analysis, 1974,
Courant Inst. Math. Sci.,N-Y Univ., AMS Providence

\bibitem{4} Krasnoselskii M. A., Zabreyko P. P., Geometrical Methods of
Nonlinear Analysis, Springer-Verlag Berlin Heidelberg 1984, XX, 412

\bibitem{5} Dunford N., Schwartz T., Linear operators, p. I, General theory,
Intersci. Publ., N-Y, London, 1958,

\bibitem{6} Kuratowski K., Topology, v I, 1966; v II, 1968, Academic Press,
N-Y and London,

\bibitem{7} Diestel J., \ Geometry of Banach Spaces-Selected Topics, 1975,
485, Lecture Notes in Math., Springer

\bibitem{8} Ells J., Foundations of the global analysis, Usp. Math. Sci.,
24,3, 1969,157-210 (in russ.)

\bibitem{9} Ells J., Fredholm structures, Usp. Math. Sci., 26, 6, 1971,
213-140 (in russ.)

\bibitem{10} Miranda C., Equazioni alle derivate parziali di tipo ellittico,
1955, Springer-Verlag

\bibitem{11} Caristi J., Fixed-point theoremsfor mappings satisfying
inwardness condition, Trans. AMS,1976, 215, 241-251

\bibitem{12} Edwards R. E., Functional Analysis: Theory and Applications,
Holt, Rinehart and Winston (1965) \ \ \ \ 

\bibitem{13} Rockafellar R. T., Convex analysis, 1970, Princeton, N-J,
Prins. Univ. Press,

\bibitem{14} Browder F. E., Petyshyn W. V., Approximation methods and the
generalized topology degree for nonlinear mappings in Banach spaces, J.
Funct. Anal. 3, 2, 1969\ \ 

\bibitem{15} Lions J-L., Quelques methodes de resolution des problemes aux
limites non lineaires, 1969, Dunod, Gauthier-Villars, Paris

\bibitem{16} Aubin J.- P., Ekeland I., Applied Nonlinear Analysis, 1984,
John Wiley and Sons, N-Y

\bibitem{17} Hammer P. C., Semispaces and the topology of convexity, Procc.
Symp. "Pure Math." , v. 7, Convexity, 1963, 305-317

\bibitem{18} Soltanov K. N., On nonlinear mappings and nonlinear equations,
Dokl. USSR, 1986, 289, 6, 1318-1323

\bibitem{19} Ekeland I., Temam R., Convex analysis and variational problems,
Study in Math. and iys Appl., 1976, Nort-Holland Publ., USA Els. C

\bibitem{20} Borisovich Yu. G., Zviagin V. G., Sapranov Yu. I., Fredholm's
nonlinear mappings and Leray-Schauder theory, Usp. Math. Sci. 1977, 32, 4
(in russ.)

\bibitem{21} Klee V. L., Convex sets in linear spaces, Duke Math. J., 1951,
18, 4, 875-893;\ 

\bibitem{22} Klee V. L., Convex sets in linear spaces, Duke Math. J. 18, 2,
1951, 444-466.

\bibitem{23} Klee V.L., Separation and support properties of convex sets a
survey, in "Control Theory and Calculus of Variat." (Ed. Balakrishnan A.
V.), Acad. Press, 1969, 235-305

\bibitem{24} Bishop E., Phelps R. R., Support functional of a convex set,
Proc. Symp. Pure Math., "Convexity", AMS, 1963,VII, 27-35

\bibitem{25} James R. C., Weakly compact sets, Trans. AMS, 1964, 113, 1,
129-140.\ 

\bibitem{26} Ioffe A. D., Tikhomirov V. M., Theory of extremal problems,
1974, Nauka, Moscow (in russ.)

\bibitem{27} Gale Y., Klee V. L., Continuous convex sets, Math. Scand.,
1959, 7, 379-391

\bibitem{28} Soltanov K. N., Remarks on Separation of Convex Sets,
Fixed-Point Theorem and Applications in Theory of Linear Operators, Fixed
Point Theory and Applic. ,V. 2007, Art. ID 80987, Hindawi P. C.

\bibitem{29} Spanier E. H., Algebraic Topology, McGRAW-Hill\ Book C. N-Y,
London, 1966

\bibitem{30} Pohozaev S. I., On nonlinear operators with the weakly closed
images and quasilinear elliptic equations, Math. Sb., 78, 2, 1969, 236-259
(in russ.)

\bibitem{31} Pohozaev S. I., On solvability of nonlinear equations with odd
operators, Funct. Anal. and its appl., 1967, 1, 3, 66-73. (in russ.)

\bibitem{32} Soltanov K. N., Some applications of the nonlinear analysis to
the nonlinear differential equations, 2002, Elm, Baku, 296 p. (in russ.)

\bibitem{33} Soltanov K. N., Nonlinear equations in Banach spaces with
continuous operators and differential equations, Funct. Anal. \& appl., 33,1
(1999)

\bibitem{34} Soltanov K. N., On nonlinear equations with continuous
mappings, Nonlinear Boundary Problems, AS Ukrain, 1990, 2, 104-109 (in russ.)

\bibitem{35} Schaefer H. H., Topological vector spaces, 1966, Macmill. C.,
N-Y, London

\bibitem{36} Vainberg M. M., Variational Methods for the investigation of
non-linear operators, M. 1959, Engl. trans. San Francisco 1964.

\bibitem{37} Asplund E., Averaged norms, Israel J. Math., 1967, 5, 227-233

\bibitem{38} Soltanov K. N., Some remarks on the solvability of a nonlinear
equations and fixed points, Nonlinear Studies, 18, 4, 2011

\bibitem{39} Graves L. M., Some mapping theorems, Duke Math. J. 17, 2, 1950,
11-114

\bibitem{40 } Soltanov K. N., Some nonlinear equations of the nonstable
filtration type and embedding theorems. Nonlinear Analysis: T.M.\&A., 2006,
65, 11.

\bibitem{41} Sobolev S. L., Some applications of functional analysis in
mathematical physics, 1962, Sib. Div. AS USSR (in russ.)

\bibitem{42} Gajewski H., Groger K, Zacharias K., Nichtlineare
operatorgleichungen und operatordifferentialgleichungen, 1974, Math.
Monogr., Akad.-Verlag, Berlin

\bibitem{43} Soltanov K. N., To theory of normal solvability of nonlinear
equations, Dokl. AS USSR, 278, 1, 1984, 42-46;

\bibitem{44} Soltanov K. N., Ahmadov M. A., Solvability of equation of
Prandtl-von Mises type, theorems of embedding, Trans. NAS Azerb., 37 (1),
(2017)
\end{thebibliography}
\end{document}